\documentclass[12pt]{amsart}
\usepackage{amssymb}
\usepackage{amsfonts}

\theoremstyle{plain}

\newtheorem{lemma}{Lemma}
\newtheorem{corollary}{Corollary}
\newtheorem{proposition}{Proposition}

\theoremstyle{definition}

\newtheorem{example}{Example}

\theoremstyle{remark}
\newtheorem{remark}{Remark}

\numberwithin{equation}{section}
\input{tcilatex}

\begin{document}
\title[On $L_{p}$-estimates of singular integrals]{On $L_{p}$-estimates of
some singular integrals related to jump processes}
\author{R. Mikulevicius and H. Pragarauskas}
\address{University of Southern California, Los Angeles\\
Institute of Mathematics and Informatics, Vilnius University}
\date{March 14, 2012}

\begin{abstract}
We estimate fractional Sobolev and Besov norms of some singular integrals
arising in the model problem for the Zakai equation with discontinuous
signal and observation.
\end{abstract}

\maketitle

\section{ Introduction}

In a complete probability space $(\Omega ,\mathcal{F},\mathbf{P})$ with a
filtration of $\sigma $-algebras $\mathbb{F}=(\mathcal{F}_{t})$ satisfying
the usual conditions, the following linear stochastic integro-differential
parabolic equation of the fixed order $\alpha \in (0,2]$ was considered in H%
\"{o}lder classes (see \cite{mikprag1}): 
\begin{equation}
\left\{ 
\begin{array}{ll}
du(t,x)=\bigl(A^{({\scriptsize \alpha )}}u(t,x)+f(t,x)\bigr)%
dt+\int_{U}g(t,x,v)q(dt,dv) & \quad \text{in }{E}_{0,T}, \\ 
u(0,x)=u_{0}(x) & \quad \text{in }\mathbf{R}^{d},%
\end{array}%
\right.  \label{intr1}
\end{equation}%
where ${E}_{0,T}=\left[ 0,T\right] \times \mathbf{R}^{d},\,f$ is an $\mathbb{%
F}$-adapted measurable real-valued function on $\mathbf{R}^{d+1}$, 
\begin{eqnarray*}
&&A^{({\scriptsize \alpha )}}u(t,x) \\
&=&\int_{\mathbf{R}_{0}^{d}}[u(t,x+y)-u(t,x)-(\nabla u(t,x),y)\chi ^{(%
{\scriptsize \alpha )}}{(y)}]m^{({\scriptsize \alpha )}}(t,y)\frac{dy}{%
|y|^{d+{\scriptsize \alpha }}} \\
&&\quad +{}\bigl(b(t),\nabla u(t,x)\bigr)1_{{\scriptsize \alpha =1}%
}+\sum_{i,j=1}^{d}B^{ij}(t)\partial _{ij}^{2}u(t,x)1_{{\scriptsize \alpha =2}%
},\quad (t,x)\in \mathbf{R}^{d+1},
\end{eqnarray*}%
$\chi ^{({\scriptsize \alpha )}}{(y)}=1_{{\scriptsize \alpha >1}%
}+1_{|y|\leqslant 1}1_{{\scriptsize \alpha =1}},m^{({\scriptsize \alpha )}%
}(t,y)$ is a bounded measurable real-valued function homogeneous in $y$ of
order zero, $\mathbf{R}_{0}^{d}=\mathbf{R}^{d}\backslash \{0\},$ $%
b(t)=(b^{1}(t),\ldots ,b^{d}(t))$ is a bounded measurable function and $%
B(t)=(B^{ij}(t))$ is a bounded symmetric non-negative definite measurable
matrix-valued function; 
\begin{equation*}
q(dt,d\upsilon )=p(dt,d\upsilon )-\Pi (d\upsilon )dt
\end{equation*}%
is a martingale measure on a measurable space $([0,\infty )\times U,\mathcal{%
B}([0,\infty ))\otimes \mathcal{U})$ ($p(dt,d\upsilon )$ is a Poisson point
measure on $([0,\infty )\times U,\mathcal{B}([0,\infty ))\otimes \mathcal{U}%
) $ with the compensator $\Pi (d\upsilon )dt)$ and $g$ is an $\mathbb{F}$%
-adapted measurable real-valued function on $\mathbf{R}^{d+1}\times U.$ It
is the model problem for the Zakai equation (see \cite{za}) arising in the
nonlinear filtering problem with discontinuous observation (see \cite%
{mikprag1}). Let us consider the following example.

\begin{example}
\label{exa1}\textit{Assume that the signal process }$X_{t}$\textit{\ in }$%
R^{d}$\textit{\ is defined by }%
\begin{equation*}
X_{t}=X_{0}+\int_{0}^{t}b(X_{s})ds+W_{t}^{\alpha },t\in \lbrack 0,T],
\end{equation*}%
\textit{where }$b(x)=(b^{i}(x))_{1\leq i\leq d}$\textit{,}$x\in R^{d},$%
\textit{\ are measurable and bounded }$W_{t}^{\alpha }$\textit{\ is a }$d$%
\textit{-dimensional }$\alpha $\textit{-stable (}$\alpha \in (1,2)$\textit{)
L\'{e}vy process. Suppose}%
\begin{equation*}
W^{\alpha }{}_{t}=\int_{0}^{t}\int \upsilon \lbrack p(ds,d\upsilon )-m\left( 
\frac{\upsilon }{|\upsilon |}\right) \frac{d\upsilon ds}{|\upsilon
|^{d+\alpha }}],
\end{equation*}%
\textit{where }$m(\frac{\upsilon }{|\upsilon |})$\textit{\ is a smooth
bounded function (it characterizes the intensity of the jumps of }$W^{\alpha
}$\textit{\ in in the direction }$\frac{\upsilon }{|\upsilon |}$\textit{)
and }$p(ds,d\upsilon )$\textit{\ is a Poisson point measure on }$[0,\infty
)\times R_{0}^{d}$\textit{\ with }%
\begin{equation*}
\mathbf{E}p(ds,d\upsilon )=m(\frac{\upsilon }{|\upsilon |})\frac{d\upsilon ds%
}{|\upsilon |^{d+\alpha }}.
\end{equation*}%
\textit{Assume }$X_{0}$\textit{\ has a density function }$u_{0}\left(
x\right) ,$\textit{\ and the observation }$Y_{t}$\textit{\ is discontinuous,
with jump intensity depending on the signal, such that}%
\begin{equation*}
Y_{t}=\int_{0}^{t}\int_{|y|>1}y\hat{p}(ds\,dy)+\int_{0}^{t}\int_{|y|%
\leqslant 1}y\hat{q}(ds,dy),
\end{equation*}%
\textit{where }$\hat{p}(ds,dy)$\textit{\ is a point measure on }$[0,\infty
)\times R_{0}^{d}$\textit{\ not having common jumps with }$W^{\alpha }$%
\textit{\ with a compensator }$\rho (X_{t},y)\pi (dy)$\textit{\ and }$\hat{q}%
(dt,dy)=\hat{p}(dt,dy)-\pi (dy)dt$\textit{. Assume }$C_{1}\geqslant \rho
(x,y)\geqslant c_{1}>0,\pi (dy)$\textit{\ is a measure on }$R_{0}^{d}$%
\textit{\ such that}%
\begin{equation*}
\int |y|^{2}\wedge 1\pi (dy)<\infty ,
\end{equation*}%
\textit{and }$\int [\rho (x,y)-1]^{2}\pi (dy)$\textit{\ is bounded. Then for
every function }$\varphi $\textit{\ such that }$E[\varphi
(X_{t})^{2}]<\infty ,$\textit{\ the optimal mean square estimate for }$%
\varphi \left( X_{t}\right) ,\,t\in \left[ 0,T\right] $\textit{, given the
past of \ the observations }$F_{t}^{Y}=\sigma (Y_{s},s\leqslant t),$\textit{%
\ is of the form }%
\begin{equation*}
\hat{\varphi}_{t}=\mathbf{E}\bigl[\varphi (X_{t})|\mathcal{F}_{t}^{Y}\bigr]=%
\frac{\mathbf{\widetilde{E}}\big[\varphi (X_{t})\zeta _{t}|\mathcal{F}%
_{t}^{Y}\big]}{\mathbf{\widetilde{E}}\big[\zeta _{t}|\mathcal{F}_{t}^{Y}\big
]},
\end{equation*}%
\textit{where }$\zeta _{t}$\textit{\ is the solution of the linear equation}%
\begin{equation*}
d\zeta _{t}=\zeta _{t-}\int [\rho (X_{t-},y)-1]\hat{q}(dt,dy)
\end{equation*}%
\textit{and }$d\widetilde{P}=\zeta \left( T\right) ^{-1}dP.$\textit{\ Under
assumptions of differentiability, one can easily show that if }$v(t,x)$%
\textit{\ is an }$F=(F_{t+}^{Y})$\textit{-adapted unnormalized filtering
density function }%
\begin{equation}
\mathbf{\widetilde{E}}\left[ \varphi \left( X_{t}\right) \zeta _{t}|\mathcal{%
F}_{t}^{Y}\right] =\int v\left( t,x\right) \psi \left( x\right) \,dx,
\label{prf0}
\end{equation}%
\textit{then it is a solution of the Zakai equation }%
\begin{eqnarray}
&&dv(t,x)  \label{prf1} \\
&=&v(t,x)\int [\rho (x,y)-1]\hat{q}(dt,dy)+\bigg\{-\partial _{i}\bigl(%
b^{i}(x)v(t,x)\bigr)  \notag \\
&&+\int_{\mathbf{R}_{0}^{d}}[v(t,x+y)-v(t,x)-(\nabla v(t,x),y)]m(\frac{-y}{%
|y|})\frac{dy}{|y|^{d+\alpha }}\bigg\},  \notag \\
v(0,x) &=&u_{0}(x).  \notag
\end{eqnarray}%
\textit{Since }$Y_{t},t\geqslant 0,$\textit{\ and }$X_{t},t\geqslant 0,$%
\textit{\ are independent with respect to }$\widetilde{P},$\textit{\ for }$%
u\left( t,x\right) =v\left( t,x\right) -u_{0}\left( x\right) $\textit{\ we
have an equation whose model problem is of the type given by (\ref{intr1}).
Indeed, according to \cite{grig1}, for any infinitely differentiable
function }$\varphi $\textit{\ on }$\mathbf{R}^{d}$\textit{\ with compact
support, the conditional expectation }$\pi _{t}(\varphi )=\widetilde{E}\left[
\varphi \left( X_{t}\right) \zeta _{t}|\mathcal{F}_{t}^{Y}\right] $\textit{\
satisfies the equation}%
\begin{eqnarray*}
d\pi _{t}(\varphi ) &=&\int \pi _{t}\big(\varphi \lbrack \rho (\cdot ,y)-1]%
\big)\hat{q}(dt,dy)+\pi _{t}\bigg\{(b,\nabla \varphi ) \\
&&+\int_{\mathbf{R}_{0}^{d}}\big[\varphi (\cdot +y)-\varphi -(\nabla \varphi
,y)\chi ^{(\alpha )}(y)\big]m(t,\frac{y}{|y|})\frac{dy}{|y|^{d+\alpha }}%
\bigg\}dt.
\end{eqnarray*}%
\textit{Assuming (\ref{prf0}) and integrating by parts, we obtain (\ref{prf1}%
).}
\end{example}

In terms of Fourier transform,%
\begin{equation*}
A^{({\scriptsize \alpha )}}v(x)=\mathcal{F}^{-1}\left[ \psi ^{({\scriptsize %
\alpha )}}(t,\xi )\mathcal{F}v(\xi )\right] (x),
\end{equation*}%
with 
\begin{eqnarray*}
{}\psi ^{({\scriptsize \alpha )}}(t,\xi ) &=&i(b(t),\xi )1_{{\scriptsize %
\alpha =1}}-\sum_{i,j=1}^{d}B^{ij}(t)\xi _{i}\xi _{j}1_{{\scriptsize \alpha
=2}} \\
&&-C\int_{S^{d-1}}|(w,\xi )|^{{\scriptsize \alpha }}\Big[1-i\Big(\tan \frac{%
\alpha \pi }{2}\mbox{sgn}(w,\xi )1_{{\scriptsize \alpha \neq 1}} \\
&&-\frac{2}{\pi }\mbox{sgn}(w,\xi )\ln |(w,\xi )|1_{{\scriptsize \alpha =1}}%
\Big)\Big]m^{(\alpha )}(t,w)dw,
\end{eqnarray*}%
where $C=C(\alpha ,d)$ is a positive constant, $S^{d-1}$ is the unit sphere
in $\mathbf{R}^{d}$ and $dw$ is the Lebesgue measure on it. It was shown in 
\cite{mikprag1} that in H\"{o}lder classes the solution of (\ref{intr1}) can
be represented as%
\begin{equation}
u(t,x)=Rf(t,x)+\widetilde{R}g(t,x)+T_{t}u_{0}(x),  \label{maf1}
\end{equation}%
where%
\begin{eqnarray}
Rf(t,x) &=&\int_{0}^{t}G_{s,t}\ast f(s,x)ds,  \notag \\
\widetilde{R}g(t,x) &=&\int_{0}^{t}\int_{U}G_{s,t}\ast g(s,x,\upsilon
)q(ds,d\upsilon ),  \label{rez} \\
T_{t}u_{0}(x) &=&G_{0,t}\ast u_{0}(x),  \notag
\end{eqnarray}%
with 
\begin{equation*}
G_{s,t}(x)=\mathcal{F}^{-1}\left( \exp \left\{ \int_{s}^{t}\psi ^{(%
{\scriptsize \alpha )}}(r,\xi )dr\right\} \right) x,\quad s\leq t,x\in 
\mathbf{R}^{d},
\end{equation*}%
and $\ast $ denoting the convolution with respect to the space variable $%
x\in \mathbf{R}^{d}.$ According to \cite{SaT94}, $G_{s,t}$ is the density
function of an $\alpha $- stable distribution, and $A^{(\alpha )}$ is the
fractional Laplacian if $b=0$ and $m^{(a)}=1$. \ 

In order to estimate the $L_{p}$-norm of the fractional derivative 
\begin{equation*}
\partial ^{{\scriptsize \alpha }}u(t,x)=-\mathcal{F}^{-1}[|\xi |^{%
{\scriptsize \alpha }}\mathcal{F}u(t,\xi )]
\end{equation*}%
of $u$ in (\ref{maf1}), we need the estimates for $\partial ^{\alpha
}Rf,\partial ^{\alpha }\tilde{R}g$ and $\partial ^{\alpha }T_{t}u_{0}.$ It
was derived in \cite{MiP922}, that%
\begin{equation*}
|\partial ^{\alpha }Rf|_{L_{p}}\leq C|f|_{L_{p}}.
\end{equation*}%
According to Corollary \ref{cor5} below (it provides two-sided estimates for
the moments of a martingale), 
\begin{equation*}
\mathbf{E}|\partial ^{\alpha }\widetilde{R}g|_{L_{p}}^{p}\leq C[\mathbf{E}%
I_{1}+\mathbf{E}I_{2}],
\end{equation*}%
where%
\begin{eqnarray}
I_{1} &=&\int_{0}^{T}\int_{\mathbf{R}^{d}}\left\{ \int_{0}^{t}\int_{U}\left[
\partial ^{{\scriptsize \alpha }}G_{s,t}\ast g(s,x,\upsilon )\right] ^{2}\Pi
(d\upsilon )ds\right\} ^{p/2}dxdt,  \label{nf1} \\
I_{2} &=&\int_{0}^{T}\int_{0}^{t}\int_{\mathbf{R}^{d}}\int_{U}|\partial ^{%
{\scriptsize \alpha }}G_{s,t}\ast g(s,x,\upsilon )|^{p}\Pi (d\upsilon
)dxdsdt.  \label{nf11}
\end{eqnarray}

In this paper, we estimate the singular integrals of $I_{1}$- and $I_{2}$%
-types related to $\widetilde{R}g(t,x)$ in (\ref{rez}) in Sobolev and Besov
spaces. If $\alpha =2$ and $B$ is $d{\times }d$-identity matrix, the
estimate of $I_{1}$-type was proved in \cite{kry1}. This estimate for (\ref%
{nf1}) was generalized in \cite{kim1} for the case $m^{({\scriptsize \alpha )%
}}=1$, $b=0$ $($in this case $A^{(a)}$ is the fractional Laplacian). Our
derivation of an estimate for (\ref{nf1}) follows a slightly different idea
communicated by N.V. Krylov. The problem cannot be reduced to a case with
fractional Laplacian. In fact, $m^{(\alpha )}$ can be zero on a substantial
set (see Remark \ref{r1}). The operator $\tilde{R}g$ in H\"{o}lder-Zygmund
classes was estimated in \cite{mikprag1}. The results of this paper were
applied in \cite{mikprag2} to solve the model problem above in the
fractional Sobolev spaces.

The paper consists of five sections. In Section 2, we introduce the notation
and state the main results. In Section 3, we derive the two-sided $p$-moment
estimates of discontinuous martingales that explain the need to consider (%
\ref{nf1}) and (\ref{nf11}). In the last two sections, we present the proofs
of the main results.

\section{Notation, function spaces and main results}

\subsection{Notation}

The following notation will be used in the paper.

Let $\mathbf{N}_{0}=\{0,1,2,\ldots \},\mathbf{R}_{0}^{d}=\mathbf{R}%
^{d}\backslash \{0\}.$ If $x,y\in \mathbf{R}^{d},$\ we write 
\begin{equation*}
(x,y)=\sum_{i=1}^{d}x_{i}y_{i},\,|x|=\sqrt{(x,x)}.
\end{equation*}

We denote by $C_{0}^{\infty }(\mathbf{R}^{d})$ the set of all infinitely
differentiable functions on $\mathbf{R}^{d}$ with compact support.

We denote the partial derivatives in $x$ of a function $u(t,x)$ on $\mathbf{R%
}^{d+1}$ by $\partial _{i}u=\partial u/\partial x_{i}$, $\partial
_{ij}^{2}u=\partial ^{2}u/\partial x_{i}\partial x_{j}$, etc.;$\,\partial
u=\nabla u=(\partial _{1}u,\ldots ,\partial _{d}u)$ denotes the gradient of $%
u$ with respect to $x$; for a multiindex $\gamma \in \mathbf{N}_{0}^{d}$ we
denote%
\begin{equation*}
\partial _{x}^{{\scriptsize \gamma }}u(t,x)=\frac{\partial ^{|{\scriptsize %
\gamma |}}u(t,x)}{\partial x_{1}^{{\scriptsize \gamma _{1}}}\ldots \partial
x_{d}^{{\scriptsize \gamma _{d}}}}.
\end{equation*}%
For $\alpha \in (0,2]$ and a function $u(t,x)$ on $\mathbf{R}^{d+1}$, we
write 
\begin{equation*}
\partial ^{{\scriptsize \alpha }}u(t,x)=-\mathcal{F}^{-1}[|\xi |^{%
{\scriptsize \alpha }}\mathcal{F}u(t,\xi )](x),
\end{equation*}%
where 
\begin{equation*}
\mathcal{F}h(t,\xi )=\int_{\mathbf{R}^{d}}\,\mathrm{e}^{-i({\scriptsize \xi
,x)}}h(t,x)dx,\mathcal{F}^{-1}h(t,\xi )=\frac{1}{(2\pi )^{d}}\int_{\mathbf{R}%
^{d}}\,\mathrm{e}^{i({\scriptsize \xi ,x)}}h(t,\xi )d\xi .
\end{equation*}

The letters $C=C(\cdot ,\ldots ,\cdot )$ and $c=c(\cdot ,\ldots ,\cdot )$
denote constants depending only on quantities appearing in parentheses. In a
given context the same letter will (generally) be used to denote different
constants depending on the same set of arguments.

\subsection{Function spaces\label{test}}

Let $\mathcal{S}(\mathbf{R}^{d})$ be the Schwartz space of smooth
real-valued rapidly decreasing functions. Let $V$ be a Banach space with a
norm $\vert \cdot\vert_V $. The space of $V$-valued tempered distributions
we denote by $\mathcal{S}^{\prime }(\mathbf{R}^{d},V)$ ($f\in \mathcal{S}%
^{\prime }(\mathbf{R}^{d},V)$ is a continuous $V$-valued linear functional
on $\mathcal{S}(\mathbf{R}^{d})$).

For a $V$-valued measurable function $h$ on $\mathbf{R}^{d}$ and $%
p\geqslant1 $ we denote 
\begin{equation*}
|h|_{V,p}^{p}=\int_{\mathbf{R}^{d}}|h(x)|_{V}^{p}dx.
\end{equation*}

Further, for a characterization of our function spaces we will use the
following construction (see \cite{BeL76}). By Lemma 6.1.7 in \cite{BeL76},
there exists a function $\phi \in C_{0}^{\infty }(\mathbf{R}^{d})$ such that 
$\mathrm{supp}\,\phi =\{\xi :\frac{1}{2}\leqslant |\xi |\leqslant 2\}$, $%
\phi (\xi )>0$ if $2^{-1}<|\xi |<2$ and 
\begin{equation*}
\sum_{j=-\infty }^{\infty }\phi (2^{-j}\xi )=1\quad \text{if }\xi \neq 0.
\end{equation*}%
Define the functions $\varphi _{k}\in \mathcal{S}(\mathbf{R}^{d}),$ $%
k=1,\ldots ,$ by 
\begin{equation*}
\mathcal{F}\varphi _{k}(\xi )=\phi (2^{-k}\xi ),
\end{equation*}%
and $\varphi _{0}\in \mathcal{S}(\mathbf{R}^{d})$ by 
\begin{equation*}
\mathcal{F}\varphi _{0}(\xi )=1-\sum_{k\geqslant 1}\mathcal{F}\varphi
_{k}(\xi ).
\end{equation*}

Let $\beta\in \mathbf{R}$ and $p\geqslant1$. We introduce the Besov space $%
B_{pp}^{{\scriptsize \beta }}=B_{pp}^{{\scriptsize \beta }}(\mathbf{R}%
^{d},V) $ of generalized functions $f\in \mathcal{S}^{\prime }(\mathbf{R}%
^{d},V)$ with finite norm%
\begin{equation*}
|f|_{B_{pp}^{{\scriptsize \beta }}(\mathbf{R}^{d},V)}=\Bigg\{%
\sum_{j=0}^{\infty }2^{j{\scriptsize \beta} p}|\varphi _{j}\ast f|_{V,p}^{p}%
\Bigg\}^{1/p},
\end{equation*}%
the Sobolev space $H_{p}^{{\scriptsize \beta }}(\mathbf{R}^{d},V)$ of $f\in 
\mathcal{S}^{\prime }(\mathbf{R}^{d},V)$ with finite norm%
\begin{eqnarray}
|f|_{H_{p}^{{\scriptsize \beta }}(\mathbf{R}^{d},V)} &=&|\mathcal{F}%
^{-1}((1+|\xi |^{2})^{{\scriptsize \beta /2}}\mathcal{F}f)|_{V,p}
\label{nf5} \\
&=&\left\vert (I-\Delta )^{{\scriptsize \beta /2}}f\right\vert _{V,p}, 
\notag
\end{eqnarray}%
where $I$ is the identity map and $\Delta $ is the Laplacian in $\mathbf{R}%
^{d}$, and the space $\tilde{H}_{p} ^{{\scriptsize \beta }}(\mathbf{R}%
^{d},V) $ of $f\in \mathcal{S}^{\prime }(\mathbf{R}^{d},V)$ with finite norm%
\begin{equation}
|f|_{\tilde{H}_{p}^{{\scriptsize \beta }}(\mathbf{R}^{d},V)}=\left\{\int_{%
\mathbf{R}^{d}}\left( \sum_{j=0}^{\infty }2^{2{\scriptsize \beta} j}|\varphi
_{j}\ast f(x)|_{V}^{2}\right) ^{p/2}dx \right\}^{1/p}.  \label{nf0}
\end{equation}

Similarly we introduce the corresponding spaces of generalized functions on $%
E_{a,b}=[a,b]\times \mathbf{R}^{d}$ and $\tilde{E}_{a,b}=\{ (s,t,x)\in 
\mathbf{R}^{d+2}:a\leq s\leq t\leq b,x\in \mathbf{R}^{d}\} . $

The spaces $B_{pp}^{{\scriptsize \beta }}(E_{a,b},V)$, $H_{p}^{{\scriptsize %
\beta }}(E_{a,b},V)$ and $\tilde{H}_{p}^{{\scriptsize \beta }}(E_{a,b},V)$
consist of all measurable $\mathcal{S}^{\prime}(\mathbf{R}^{d},V)$-valued
functions on $[a,b]$ with finite corresponding norms:%
\begin{eqnarray}
|f|_{B_{pp}^{{\scriptsize \beta }}(E_{a,b},V)}= \Bigg\{\int_{a}^{b}|
f(t,\cdot )|_{B_{pp}^{{\scriptsize \beta }}(\mathbf{R}^{d},V)}^{p}dt\Bigg\}%
^{1/p} ,  \notag \\
|f|_{H_{p}^{{\scriptsize \beta }}(E_{a,b},V)}= \Bigg\{\int_{a}^{b}|
f(t,\cdot )|_{H_{p}^{{\scriptsize \beta }}(\mathbf{R}^{d},V)}^{p}dt\Bigg\}%
^{1/p}  \label{norm11}
\end{eqnarray}
and 
\begin{equation}
|f|_{\tilde{H}_{p}^{{\scriptsize \beta }}(E_{a,b},V)}= \Bigg\{\int_{a}^{b}|
f(t,\cdot )|_{\tilde H_{p}^{{\scriptsize \beta }}(\mathbf{R}^{d},V)}^{p}dt%
\Bigg\}^{1/p}.  \label{norm2}
\end{equation}

The spaces $B_{pp}^{{\scriptsize \beta }}(\tilde E_{a,b},V)$, $H_{p}^{%
{\scriptsize \beta }}(\tilde E_{a,b},V)$ and $\tilde{H}_{p}^{{\scriptsize %
\beta }}(\tilde E_{a,b},V)$ consist of all measurable $\mathcal{S}^{\prime}(%
\mathbf{R}^{d},V)$-valued functions on $\{(s,t): a\leqslant s\leqslant
t\leqslant b\}$ with finite corresponding norms:%
\begin{eqnarray}
|f|_{B_{pp}^{{\scriptsize \beta }}(\tilde E_{a,b},V)}= \Bigg\{%
\int_{a}^{b}\int_{a}^{t}| f(s,t,\cdot )|_{B_{pp}^{{\scriptsize \beta }}(%
\mathbf{R}^{d},V)}^{p}dsdt\Bigg\}^{1/p} ,  \notag \\
|f|_{H_{p}^{{\scriptsize \beta }}(\tilde E_{a,b},V)}= \Bigg\{%
\int_{a}^{b}\int_{a}^{t}| f(s,t,\cdot )|_{H_{p}^{{\scriptsize \beta }}(%
\mathbf{R}^{d},V)}^{p}dsdt\Bigg\}^{1/p}  \label{norm11n}
\end{eqnarray}
and 
\begin{equation}
|f|_{\tilde{H}_{p}^{{\scriptsize \beta }}(\tilde E_{a,b},V)}= \Bigg\{%
\int_{a}^{b}\int_{a}^{t}| f(s,t,\cdot )|_{\tilde H_{p}^{{\scriptsize \beta }%
}(\mathbf{R}^{d},V)}^{p}dsdt\Bigg\}^{1/p}.  \label{norm2n}
\end{equation}

For the scalar functions the norms (\ref{nf5}) and (\ref{nf0}) are
equivalent (see \cite{Tri92}, p.~15). Therefore, the norms (\ref{norm11})
and (\ref{norm2}) as well as (\ref{norm11n}) and (\ref{norm2n}) are
equivalent.

If $V$ is a separable Hilbert space, we will also use the spaces $\bar
B_{pp}^{{\scriptsize \beta}}(\tilde E_{a,b},V)$ and $\bar H_{p}^{%
{\scriptsize \beta}}(\tilde E_{a,b},V)$ consisting of measurable $S^{\prime
}(\mathbf{R}^d,V)$-valued functions on $\{(s,t):a\leq s\leq t\leq b\}$ with
finite norms 
\begin{equation*}
\vert f\vert_{\bar B_{pp}^{{\scriptsize \beta}}(\tilde E_{a,b},V)} = %
\biggl\{ \sum_{j=0}^{\infty}2^{j{\scriptsize \beta p}}\int_a^b\int_{\mathbf{R%
}^d} \biggl( \int_a^t\big\vert \varphi_j\ast f(s,t,x)\big\vert^2_V ds \biggr)%
^{p/2}dxdt \biggr\}^{1/p}
\end{equation*}
and 
\begin{equation*}
\vert f\vert_{\bar H_{p}^{{\scriptsize \beta}}(\tilde E_{a,b},V)} = \biggl\{ %
\int_a^b\int_{\mathbf{R}^d} \biggl( \int_a^t\big\vert \mathcal{F}^{-1}\bigl( %
(1+\vert \xi\vert^2 )^{{\scriptsize \beta/2}}\mathcal{F}f\bigr)(s,t,x) %
\big\vert^2_V ds \biggr)^{p/2}dxdt \biggr\}^{1/p} .
\end{equation*}

\subsection{Main results}

Throughout the paper we assume that the functions $b=b(t),B=B(t)$ and $m^{(%
{\scriptsize \alpha )}}(t,y)\geq 0$ are measurable, $m^{(2)}=0$ and%
\begin{equation*}
\int_{S^{d-1}}wm^{(1)}(t,w)dw=0,t\in \mathbf{R}.
\end{equation*}

Also, we will need the following assumptions.

\ \textbf{A}. (i) The function $m=m(t,y)\geq 0$ is 0-homogeneous and
differentiable in $y$ up to $d_{0}=\left[ \frac{d}{2}\right] +1;$

(ii) There is a constant $K$ such that for each $\alpha \in (0,2)$ and $t\in 
\mathbf{R}$%
\begin{equation*}
|b(t)|+|B(t)|+\sup_{\substack{ |\gamma |\leq d_{0},  \\ |\xi |=1}}|\partial
_{y}^{{\scriptsize \gamma }}m^{({\scriptsize \alpha )}}(t,y)|\leq K.
\end{equation*}

\textbf{B.} There is a constant $\mu >0$ such that%
\begin{equation*}
\sup_{t,|\xi |=1}\mathop{\mathrm{Re}}\psi ^{({\scriptsize \alpha )}}(t,\xi
)\leq -\mu .
\end{equation*}

\begin{remark}
\label{r1}The assumption \textbf{B} holds with certain $\mu >0$ if, for
example,%
\begin{eqnarray*}
\inf_{t,|\xi |=1}(B(t)\xi ,\xi ) &>&0\text{, }\alpha =2, \\
\inf_{t,w\in \Gamma }m^{({\scriptsize \alpha )}}(t,w) &>&0,\alpha \in (0,2),
\end{eqnarray*}%
for a measurable subset $\Gamma \subseteq S^{d-1}$ of a positive Lebesgue
measure.
\end{remark}

Given a measurable $\mathcal{S}^{\prime }(\mathbf{R}^{d},V)$-valued function 
$g$ on $\mathbf{R}$, we consider a linear operator $\mathcal{I}$ that
assigns to it a $\mathcal{S}^{\prime }(\mathbf{R}^{d},V)$-valued function on 
$\{(s,t):s\leq t\}:$%
\begin{equation*}
\mathcal{I}g(s,t,x)=G_{s,t}\ast g(s,x),s\leq t,x\in \mathbf{R}^{d}.
\end{equation*}

The main results of the paper are the two propositions given below.
Proposition \ref{main1} in the case $V=L_p(U,\mathcal{U},\Pi)$ is related to
the integral $I_2$ in (\ref{nf11}) and Proposition \ref{main2} in the case $%
V=L_2(U,\mathcal{U},\Pi)$ is related to the integral $I_1 $ in (\ref{nf1}).


\begin{proposition}
\label{main1}Let Assumptions $\mathbf{A}\mbox{ and }{B}$ hold, $p\geq
2,\beta \in \mathbf{R,}-\infty \leq a<b\leq \infty $. Then the operator $%
\mathcal{I}:B_{pp}^{{\scriptsize \beta +\alpha -\frac{\alpha }{p}}%
}(E_{a,b},V)\rightarrow \tilde{H}_{p}^{{\scriptsize \beta +\alpha }}(\tilde{E%
}_{a,b},V)$ is bounded: there is a constant $C=C(\alpha ,K,\mu ,p,d)$ such
that 
\begin{equation}
|\mathcal{I}g|_{\tilde{H}_{p}^{{\scriptsize \beta +\alpha }}(\tilde{E}%
_{a,b},V)}\leq C|g|_{B_{pp}^{{\scriptsize \beta +\alpha -\frac{\alpha }{p}}%
}(E_{a,b},V)},g\in B_{pp}^{{\scriptsize \beta +\alpha -\frac{\alpha }{p}}%
}(E_{a,b},V).  \label{nf2}
\end{equation}
\end{proposition}

Since for the scalar functions the norms (\ref{norm11n}) and (\ref{norm2n})
are equivalent, we have the following statement.

\begin{corollary}
\label{corn0}Let $V=L_{p}(U,\mathcal{U},\Pi )$. Then Proposition~\ref{main1}
holds with $\tilde{H}_{p}^{{\scriptsize \beta +\alpha }}(\tilde{E}_{a,b},V)$
replaced by $H_{p}^{{\scriptsize \beta +\alpha }}(\tilde{E}_{a,b},V)$.
\end{corollary}

\begin{proof}
Let $V=L_{p}(U,\mathcal{U},\Pi ).$ If $\mathcal{I}g\in H_{p}^{\beta +\alpha
}(\tilde{E}_{a,b},V)$, then $\Pi $ -a.e. $\mathcal{I}g(\cdot ,\cdot ,v)\in
H_{p}^{{\scriptsize \beta +\alpha }}(\tilde{E}_{a,b},\mathbf{R}).$ Since the
norms (\ref{norm11}) and (\ref{norm2}) are equivalent for the scalar
functions, we have%
\begin{eqnarray*}
|\mathcal{I}g|_{H_{p}^{\beta +\alpha }(\tilde{E}_{a,b},V)}^{p}
&=&\int_{a}^{b}\int_{\mathbf{R}^{d}}\int_{U}|(I-\Delta )^{(\beta +\alpha )/2}%
\mathcal{I}g(t,x,\upsilon )|^{p}\Pi (d\upsilon )dxdt \\
&\leq &C\int_{a}^{b}\int_{U}\int_{\mathbf{R}^{d}}\left( \sum_{j=0}^{\infty
}2^{2(\beta +\alpha )j}|\varphi _{j}\ast \mathcal{I}g(t,x,\upsilon
)|^{2}\right) ^{p/2}dx\Pi (d\upsilon )dt,
\end{eqnarray*}%
and by Minkowski inequality%
\begin{eqnarray*}
&&\int_{a}^{b}\int_{U}\int_{\mathbf{R}^{d}}\left( \sum_{j=0}^{\infty
}2^{2(\beta +\alpha )j}|\varphi _{j}\ast \mathcal{I}g(t,x,\upsilon
)|^{2}\right) ^{p/2}dx\Pi (d\upsilon )dt \\
&\leq &C\int_{a}^{b}\int_{\mathbf{R}^{d}}\left( \sum_{j=0}^{\infty
}2^{2(\beta +\alpha )j}|\varphi _{j}\ast \mathcal{I}g(t,x,\cdot
)|_{V}^{2}\right) ^{p/2}dxdt \\
&=&C|\mathcal{I}g|_{\tilde{H}_{p}^{\beta +\alpha }(\tilde{E}_{a,b},V)}^{p}
\end{eqnarray*}%
and the statement follows by Proposition \ref{main1}.
\end{proof}

\begin{proposition}
\label{main2}Let Assumptions $\mathbf{A}$ (with $d_{0}$ replaced by $d_{0}+1$%
) and $\mathbf{B}$ hold, $p\geq 2,\beta \in \mathbf{R},-\infty \leq a<b\leq
\infty $, and let $V$ be a separable Hilbert space.

Then there is a constant $C=C(\alpha ,K,\mu ,p,d)$ such that%
\begin{equation*}
|\mathcal{\partial }^{{\scriptsize \alpha /2}}\mathcal{I}g|_{\bar H_{p}^{%
{\scriptsize \beta }}(\tilde{E}_{a,b},V)}\leq C|g|_{H_{p}^{{\scriptsize %
\beta }}(E_{a,b},V)},\quad g\in H^{{\scriptsize \beta}}_p(E_{a,b},V)
\end{equation*}
and%
\begin{equation*}
|\partial ^{{\scriptsize \alpha /2}}\mathcal{I}g|_{\bar B_{pp}^{{\scriptsize %
\beta }}(\tilde{E}_{a,b},V)}\leq C|g|_{B_{pp}^{{\scriptsize \beta }%
}(E_{a,b},V)},\quad g\in B^{{\scriptsize \beta}}_{pp}(E_{a,b},V).
\end{equation*}
\end{proposition}

\section{Moment estimates of discontinuous martingales}

The following two-sided moment estimate for discontinuous martingales should
be well known (see e.g. \cite{PrT97} for this type of estimate from above).
For the sake of completeness we provide its proof. Let $p(dt,d\upsilon )$ be
a $\sigma $-finite point measure on $([0,\infty )\times U,\mathcal{B}%
([0,\infty ))\otimes \mathcal{U})$ with a dual predictable projection
measure $\pi (dt,d\upsilon )$ such that $\pi \left( \{t\}\times U\right)
=0,t\geq 0$, and let $\mathcal{R}(\mathbb{F})$ be the progressive $\sigma $%
-algebra on $[0,\infty )\times \Omega $ (see \cite{Jac79}). Denote by $%
L_{loc}^{2}$ the space of all $\mathcal{R}(\mathbb{F})\otimes \mathcal{U}$%
-measurable functions $g(t,\upsilon )=g(\omega ,t,\upsilon )$ such that $%
\mathbf{P}$-a.s. 
\begin{equation*}
\int_{0}^{t}\int_{U}g(s,\upsilon )^{2}\pi (ds,d\upsilon )<\infty
\end{equation*}%
for all $t.$

\begin{lemma}
\label{hl0}Let $p\geq 2,g\in L_{loc}^{2}$ and%
\begin{equation*}
Q_{t}=\int_{0}^{t}\int_{U}g(s,\upsilon )q(ds,d\upsilon ),t\geq 0.
\end{equation*}%
Then there are constants $C=C(p)$ and $c=c(p)>0$ such that for any $\mathbb{F%
}$-stopping time $\tau \leq T$ 
\begin{eqnarray}
&&c\mathbf{E}\bigg[\int_{0}^{{\scriptsize \tau }}\int_{U}|g(s,\upsilon
)|^{p}\pi (d\upsilon ,ds)+\bigg(\int_{0}^{{\scriptsize \tau }%
}\int_{U}g(s,\upsilon )^{2}\pi (d\upsilon ,ds)\bigg)^{p/2}\bigg]  \notag \\
&&\quad \leq \mathbf{E}\big[\sup_{t\leq {\scriptsize \tau }}|Q_{t}|^{p}\big]
\\
&&\quad \leq C\mathbf{E}\bigg[\int_{0}^{{\scriptsize \tau }%
}\int_{U}|g(s,\upsilon )| ^{p}\pi(d\upsilon ,ds)+\bigg(\int_{0}^{%
{\scriptsize \tau }}\int_{U}g(s,\upsilon )^{2}\pi (d\upsilon ,ds)\bigg)^{p/2}%
\bigg]  \notag  \label{eight}
\end{eqnarray}
\end{lemma}

\begin{proof}
Let%
\begin{equation*}
A_{t}=\int_{0}^{t}\int_{U}g(s,\upsilon )^{2}p(ds,d\upsilon ),\quad
L_{t}=\int_{0}^{t}\int_{U}g(s,\upsilon )^{2}\pi (d\upsilon ,ds),\quad t\geq
0.
\end{equation*}%
By the Burkholder--Davis--Gundy inequality (see \cite{Jac79}), there are
positive constants $c_{p}$ and $C_{p}$ such that for each $\mathbb{F}$%
-stopping time $\tau $%
\begin{equation*}
c_{p}\mathbf{E}[A_{{\scriptsize \tau }}^{p/2}]\leq \mathbf{E}\big[%
\sup_{t\leq {\scriptsize \tau }}|Q_{t}|^{p}\big]\leq C_{p}\mathbf{E}[A_{%
{\scriptsize \tau }}^{p/2}].
\end{equation*}%
Denoting $q=p/2\geq 1$, we have%
\begin{equation*}
A_{{\scriptsize \tau }}^{q}=\sum_{s\leq {\scriptsize \tau }}\big[%
(A_{s-}+\Delta A_{s})^{q}-A_{s-}^{q}\big]=\int_{0}^{{\scriptsize \tau }%
}\int_{U}\big[(A_{s-}+g(s,\upsilon )^{2})^{q}-A_{s-} ^{q}\big]p(ds,d\upsilon
)
\end{equation*}%
and%
\begin{equation*}
\mathbf{E}[A_{{\scriptsize \tau }}^{q}]=\mathbf{E}\int_{0}^{{\scriptsize %
\tau }}\int_{U}\big[(A_{s-}+g(s,\upsilon )^{2})^{q}-A_{s-}^{q}\big]\pi
(d\upsilon ,ds).
\end{equation*}%
Since there are two positive constants $c,C$ such that for all non-negative
numbers $a,b$%
\begin{equation*}
C\bigl(b^{q}+a^{q-1}b\bigr)\geq (a+b)^{q}-a^{q}\geq c\bigl(b^{q}+a^{q-1}b%
\bigr),
\end{equation*}%
we have%
\begin{eqnarray}
&&C\mathbf{E}\int_{0}^{{\scriptsize \tau }}\int_{U}\big[|g(s,\upsilon
)|^{p}+A_{s-}^{q-1}g(s,\upsilon )^{2}\big]\pi (d\upsilon ,ds)\geq \mathbf{E}%
[A_{{\scriptsize \tau }}^{q}]  \label{hf0} \\
&&\qquad \geq c\mathbf{E}\int_{0}^{{\scriptsize \tau }}\int_{U}\big[%
|g(s,\upsilon )|^{p}+A_{s-}^{q-1}g(s,\upsilon )^{2}\big]\pi (d\upsilon ,ds).
\notag
\end{eqnarray}%
Hence, 
\begin{eqnarray*}
&&c\mathbf{E}\int_{0}^{{\scriptsize \tau }}\int_{U}|g(s,\upsilon )|^{p}\pi
(d\upsilon ,ds) \leq \mathbf{E}[A_{{\scriptsize \tau }}^{q}] \\
&&\quad \leq C\mathbf{E}\bigg\{\int_{0}^{{\scriptsize \tau }%
}\int_{U}|g(s,\upsilon )|^{p}\pi (d\upsilon ,ds) + A_{{\scriptsize \tau }%
}^{q-1}L_{{\scriptsize \tau }}\bigg\}.
\end{eqnarray*}%
On the other hand, for $q>1$,%
\begin{equation*}
L_{{\scriptsize \tau }}^{q}=q\int_{0}^{{\scriptsize \tau }}L_{s}^{q-1}dL_{s}
\end{equation*}%
and%
\begin{equation*}
\mathbf{E}[L_{{\scriptsize \tau }}^{q}]=q\mathbf{E}\int_{0}^{{\scriptsize %
\tau }}L_{s}^{q-1}dA_{s}\leq q\mathbf{E}[L_{{\scriptsize \tau }}^{q-1}A_{%
{\scriptsize \tau }}].
\end{equation*}%
According to Young's inequality, for each $\varepsilon >0$ there is a
constant $C_{{\scriptsize \varepsilon }}$ such that 
\begin{eqnarray*}
A_{{\scriptsize \tau }}^{q-1}L_{{\scriptsize \tau }} &\leq &{}\varepsilon A_{%
{\scriptsize \tau }}^{q}+C_{{\scriptsize \varepsilon }}L_{{\scriptsize \tau }%
}^{q}, \\
L_{{\scriptsize \tau }}^{q-1}A_{{\scriptsize \tau }} &\leq &{}\varepsilon L_{%
{\scriptsize \tau }}^{q}+C_{{\scriptsize \varepsilon }}A_{{\scriptsize \tau }%
}^{q}.
\end{eqnarray*}%
Therefore, there is a constant $C$ such that%
\begin{eqnarray*}
\mathbf{E[}L_{{\scriptsize \tau }}^{q}] &\leq &C\mathbf{E[}A_{{\scriptsize %
\tau }}^{q}], \\
\mathbf{E}[A_{{\scriptsize \tau }}^{q}] &\leq &C\mathbf{E}\bigg\{\int_{0}^{%
{\scriptsize \tau }}\int_{U}|g(s,\upsilon )|^{p}\pi (d\upsilon ,ds)+L_{%
{\scriptsize \tau }}^{q}\bigg\}, \\
\mathbf{E}[A_{{\scriptsize \tau }}^{q}] &\geq &\mathbf{E}\int_{0}^{%
{\scriptsize \tau }}\int_{U} |g(s,\upsilon )|^{p}\pi (d\upsilon ,ds) ,
\end{eqnarray*}%
and the statement follows.
\end{proof}

\begin{corollary}
\label{cor5}Let $p\geq 2,g=g(s,x,\upsilon )$ be such that $\mathbf{P}$-a.s.%
\begin{equation*}
\int_{0}^{T}\int_{U}\int_{\mathbf{R}^{d}}g(s,x,\upsilon )^{2}\pi (d\upsilon
,ds)dx<\infty ,
\end{equation*}%
and%
\begin{equation*}
Q(t,x)=\int_{0}^{t}\int_{U}g(s,x,\upsilon )q(ds,d\upsilon ),0\leq t\leq T.
\end{equation*}%
Then%
\begin{eqnarray*}
\mathbf{E}\sup_{s\leq {\scriptsize \tau }}|Q(s,\cdot)|_{p}^{p} &\sim &%
\mathbf{E}\Bigg\{\int_{0}^{{\scriptsize \tau }}\int_{U}|g(s,\cdot,\upsilon
)|_{p}^{p}\pi (d\upsilon ,ds)+ \\
&&+\bigg\vert \bigg[\int_{0}^{{\scriptsize \tau }}\int_U g(s,\cdot ,\upsilon
)^{2}\pi (d\upsilon ,ds)\bigg]^{1/2}\bigg\vert _{p}^{p}\Bigg\}
\end{eqnarray*}%
and%
\begin{eqnarray*}
\mathbf{E}\int_{0}^{T}|Q(s,\cdot)|_{p}^{p}ds &\sim& \mathbf{E}%
\int_{0}^{T}\sup_{s\leq t}|Q(s,\cdot)|_{p}^{p}dt \\
&\sim & \mathbf{E}\Bigg\{\int_{0}^{T}\int_{0}^{t}\int_{U}|g(s,\cdot
,\upsilon )|_{p}^{p}\pi (d\upsilon ,ds)dt+ \\
&&\quad +\int_{0}^{T}\bigg\vert \bigg[\int_{0}^{t}\int_U g(s,\cdot ,\upsilon
)^{2}\pi (d\upsilon ,ds)\bigg]^{1/2}\bigg\vert _{p}^{p}dt\Bigg\} ,
\end{eqnarray*}%
where$\vert f\vert^p_p=\int\vert f(x)\vert^pdx $ and $\sim $ denotes the
equivalence of norms.
\end{corollary}

\section{Proof of Proposition \protect\ref{main1}}

Let us introduce the functions 
\begin{eqnarray*}
\widetilde{\varphi }_{0} &=&{}\varphi _{0}+\varphi _{1}, \\
\widetilde{\varphi }_{j} &=&{}\varphi _{j-1}+\varphi _{j}+\varphi
_{j+1},\quad j\geqslant 1,
\end{eqnarray*}%
where $\varphi _{j},j\geq 0,$ are defined in Subsection \ref{test}. Let%
\begin{equation*}
h_{s,t}^{j}(x)=\mathcal{F}^{-1} \biggl\{ \exp{\biggl\{ \int_s^t \psi^{(%
{\scriptsize \alpha )}}(r,\xi)dr\biggr\} } \mathcal{F}\widetilde{\varphi }%
_{j}(\xi)\biggr\}(x),\quad j\geqslant 0.
\end{equation*}

According to Lemma 12 in \cite{mikprag1} or inequality (36) and Lemma 16 in 
\cite{MiP09}, there are constants $C$, $c>0$ such that for all $s\leq
t,j\geq 1,$%
\begin{eqnarray}
\int \big\vert h_{s,t}^{j}(x)\big\vert dx &\leq &Ce^{-c2^{j{\scriptsize %
\alpha }} (t-s)}\sum_{k\leq d_{0}}\big[2^{j{\scriptsize \alpha }}(t-s)\big]%
^{k},  \label{ff1} \\
\int |h_{s,t}^{0}(x)|dx &\leq &C.  \notag
\end{eqnarray}

For $g\in B_{pp}^{{\scriptsize \alpha -\frac{\alpha }{p}}}(E_{a,b},V),$ we
set%
\begin{equation*}
g_{j}(t,\cdot )=g(t,\cdot )\ast \varphi _{j},\quad j\geqslant 0.
\end{equation*}%
Obviously, 
\begin{equation*}
{}{}\varphi _{j}\ast \mathcal{I}g(s,t,\cdot )= \mathcal{I}g_j(s,t,\cdot
),\quad j\geqslant 0.
\end{equation*}

Since $\varphi _{j}=\varphi _{j}\ast \widetilde{\varphi }_{j},j\geq 0,$ we
have 
\begin{equation*}
\mathcal{I}g_{j}(s,t,x)=h_{s,t}^{j}\ast g_{j}(s,x).
\end{equation*}%
Therefore, by Minkowski's inequality,%
\begin{eqnarray*}
|\mathcal{I}g|_{\tilde{H}_{p}^{{\scriptsize \beta }}(\tilde{E}_{a,b},V)}^{p}
&=&\int_{a}^{b}\int_{a}^{t}\int \bigg( \sum_{j=0}^{\infty }2^{2{\scriptsize %
\beta} j}\big\vert \varphi_j\ast \mathcal{I}g(s,t,x)\big\vert_V^{2}\bigg) %
^{p/2}dxdsdt \\
&=&\int_{a}^{b}\int_{a}^{t}\int \bigg( \sum_{j=0}^{\infty }2^{2{\scriptsize %
\beta} j}\big\vert h_{s,t}^{j}\ast g_{j}(s,x)\big\vert _V^{2}\bigg) %
^{p/2}dxdsdt \\
&\leq &\int_{a}^{b}\int_{a}^{t}\Bigg( \sum_{j=0}^{\infty }2^{2{\scriptsize %
\beta} j}\bigg\{\int \big\vert h_{s,t}^{j}\ast g_{j}(s,x)\big\vert _V^{p}dx%
\bigg\}^{2/p}\Bigg) ^{p/2}dsdt.
\end{eqnarray*}%
By (\ref{ff1}),%
\begin{eqnarray*}
\bigg\{\int |h_{s,t}^{j}\ast g_{j}(s,x)|_{V}^{p}dx\bigg\}^{1/p} &\leq
&\int\vert h_{s,t}^{j}(x)\vert dx \vert g_{j}(s,\cdot)\vert _{V,p} \\
&\leq &Ce^{-c2^{{\scriptsize \alpha j}}(t-s)}|g_{j}(s,\cdot )|_{V,p},\quad
j\geq 0.
\end{eqnarray*}%
So,%
\begin{eqnarray}
|\mathcal{I}g|_{\tilde{H}_{p}^{{\scriptsize \beta }}(\tilde{E}_{a,b},V)}^{p}
&\leq &\int_{a}^{b}\int_{a}^{t}\Bigg( \sum_{j=0}^{\infty }2^{2{\scriptsize %
\beta} j}\{\int |h_{s,t}^{j}\ast g_{j}(s,x)|_{V}^{p}dx\}^{2/p}\Bigg) %
^{p/2}dsdt  \notag \\
&\leq &C\int_{a}^{b}\int_{a}^{t}\Bigg( \sum_{j=0}^{\infty }e^{-c2^{%
{\scriptsize \alpha j}}(t-s)}2^{2{\scriptsize \beta} j}|g_{j}(s,\cdot
)|_{V,p}^{2}\Bigg) ^{p/2}dsdt  \label{ff2} \\
&=&C\int_{a}^{b}\int_{s}^{b}\Bigg( \sum_{j=0}^{\infty }e^{-c2^{{\scriptsize %
\alpha j}}(t-s)}2^{2{\scriptsize \beta} j}|g_{j}(s,\cdot )|_{V,p}^{2}\Bigg) %
^{p/2}dtds.  \notag
\end{eqnarray}%
If $p=2$, we have immediately%
\begin{eqnarray*}
|\mathcal{I}g|_{H_{2}^{{\scriptsize \beta }}(\tilde{E}_{a,b},V)}^{2} &\leq
&C\int_{a}^{b}\int_{s}^{b}\sum_{j=0}^{\infty }e^{-c2^{{\scriptsize \alpha j}%
}(t-s)}2^{2{\scriptsize \beta} j}|g_{j}(s,\cdot )|_{V,2}^{2}dtds \\
&\leq &C\int_{a}^{b}\sum_{j=0}^{\infty }2^{2{\scriptsize \beta} j}2^{-%
{\scriptsize \alpha j}}|g_{j}(s,\cdot )|_{V,2}^{2}ds .
\end{eqnarray*}%
If $p>2$, we split the sum in (\ref{ff2}) as follows:%
\begin{eqnarray*}
\sum_{j=0}^{\infty }e^{-c2^{{\scriptsize \alpha j}}(t-s)}2^{2{\scriptsize %
\beta} j}|g_{j}(s,\cdot)|_{V,p}^{2} &=&\sum_{j\in J}e^{-c2^{{\scriptsize %
\alpha j}}(t-s)}2^{2{\scriptsize \beta} j}|g_{j}(s,\cdot)|_{V,p}^{2} \\
&&\quad +\sum_{j\in \mathbf{N}_0\setminus J}e^{-c2^{{\scriptsize \alpha j}%
}(t-s)}2^{2{\scriptsize \beta} j}|g_{j}(s,\cdot)|_{V,p}^{2} = A(s,t)+B(s,t),
\end{eqnarray*}%
where $J=\{j\in\mathbf{N}_0:2^{{\scriptsize \alpha }j}(t-s)\leq 1\}$.

Fix $\kappa \in (0,\frac{2\alpha }{p}).$ Using H\"{o}lder's inequality, we
get 
\begin{eqnarray*}
A(s,t) &\leq &\sum_{j\in J}2^{2{\scriptsize \beta }j}2^{{\scriptsize \kappa j%
}}2^{-{\scriptsize \kappa j}}|g_{j}(s,\cdot )|_{V,p}^{2} \\
&\leq &\bigg(\sum_{j\in J}2^{q{\scriptsize \kappa j}}\bigg)^{1/q}\bigg(%
\sum_{j\in J}2^{p{\scriptsize \beta j}}2^{-p{\scriptsize \kappa j/2}%
}|g_{j}(s,\cdot )|_{V,p}^{p}\bigg)^{2/p}
\end{eqnarray*}%
with $q=\frac{p}{p-2}$. Since 
\begin{equation*}
\sum_{j\in J}2^{q{\scriptsize \kappa j}}\leq C(t-s)^{-q{\scriptsize \kappa
/\alpha }},
\end{equation*}%
we have 
\begin{eqnarray*}
A(s,t) &\leq &C(t-s)^{-\frac{{\scriptsize \kappa }}{{\scriptsize \alpha }}}%
\bigg(\sum_{j\in J}2^{p{\scriptsize \beta j}}2^{-p{\scriptsize \kappa j/2}%
}|g_{j}(s,\cdot )|_{V,p}^{p}\bigg)^{2/p} \\
&=&C(t-s)^{-\frac{{\scriptsize \kappa }}{{\scriptsize \alpha }}}\bigg(%
\sum_{j=0}^{\infty }1_{\left\{ (t-s)\leq 2^{-{\scriptsize \alpha j}}\right\}
}2^{p{\scriptsize \beta j}}2^{-p{\scriptsize \kappa j/2}}|g_{j}(s,\cdot
)|_{V,p}^{p}\bigg)^{2/p}.
\end{eqnarray*}%
So,%
\begin{eqnarray*}
\int_{a}^{b}\int_{s}^{b}A(s,t)^{p/2}dtds &\leq
&C\int_{a}^{b}\sum_{j=0}^{\infty }2^{p{\scriptsize \beta j}}2^{-p%
{\scriptsize \kappa j/2}}|g_{j}(s,\cdot )|_{V,p}^{p}\int_{s}^{s+2^{-%
{\scriptsize \alpha j}}}(t-s)^{-\frac{p{\scriptsize \kappa }}{2{\scriptsize %
\alpha }}}dtds \\
&\leq &C\int_{a}^{b}\sum_{j=0}^{\infty }2^{-{\scriptsize \alpha j}}2^{p%
{\scriptsize \beta j}}|g_{j}(s,\cdot )|_{V,p}^{p}ds=C|g|_{B_{pp}^{%
{\scriptsize \beta -\frac{\alpha }{p}}}(E_{a,b},V)}^{p}.
\end{eqnarray*}

Let us consider the sum $B(s,t)$. By H\"{o}lder's inequality, 
\begin{equation*}
B(s,t)\leq \bigg\{\sum_{j\in \mathbf{N}_{0}\setminus J}e^{-c2^{{\scriptsize %
\alpha j}}(t-s)}\bigg\}^{\frac{1}{q}}\bigg\{\sum_{j\in \mathbf{N}%
_{0}\setminus J}e^{-c2^{{\scriptsize \alpha j}}(t-s)}2^{{\scriptsize \beta pj%
}}|g_{j}(s,\cdot )|_{V,p}^{p}\bigg\}^{\frac{2}{p}}
\end{equation*}%
with $q=\frac{p}{p-2}$. Since $e^{-c2^{{\scriptsize \alpha j}}(t-s)}$ is
decreasing in $j$, 
\begin{equation*}
\sum_{j\in \mathbf{N}_{0}\setminus J}e^{-c2^{{\scriptsize \alpha j}%
}(t-s)}\leq \int_{\left\{ 2^{{\scriptsize \alpha r}}(t-s)\geq 1\right\}
}e^{-c2^{-{\scriptsize \alpha }}2^{{\scriptsize \alpha r}}(t-s)}dr\leq C.
\end{equation*}%
Therefore, 
\begin{eqnarray*}
\int_{a}^{b}\int_{s}^{b}B(s,t)^{p/2}dtds &\leq
&C\int_{a}^{b}\sum_{j=0}^{\infty }\int_{s}^{b}e^{-c2^{{\scriptsize \alpha j}%
}(t-s)}dt2^{{\scriptsize \beta pj}}|g_{j}(s,\cdot )|_{V,p}^{p}ds \\
&\leq &C\int_{a}^{b}\sum_{j=0}^{\infty }2^{-{\scriptsize \alpha j}}2^{%
{\scriptsize \beta pj}}|g_{j}(s,\cdot )|_{V,p}^{p}ds.
\end{eqnarray*}%
Finally,%
\begin{eqnarray*}
|\mathcal{I}g|_{\tilde{H}_{p}^{{\scriptsize \beta }}(\tilde{E}_{a,b},V)}^{p}
&\leq &C\bigg[\int_{a}^{b}\int_{s}^{b}A(s,t)^{p/2}dtds+\int_{a}^{b}%
\int_{s}^{b}B(s,t)^{p/2}dtds\bigg] \\
&\leq &C\int_{a}^{b}\sum_{j=0}^{\infty }2^{-{\scriptsize \alpha j}}2^{%
{\scriptsize \beta pj}}|g_{j}(s,\cdot )|_{V,p}^{p}ds\leq C|g|_{B_{pp}^{%
{\scriptsize \beta -\frac{\alpha }{p}}}(E_{a,b},V)}^{p}.
\end{eqnarray*}

The proposition is proved.

\section{Proof of Proposition \protect\ref{main2}}

In the proof we follow an idea communicated by N.V. Krylov.

\subsection{Auxiliary results}

We start with

\begin{lemma}
\label{auxl2}Let $\delta \in (0,1),l\in (-d,\delta )$. Assume that a
function $F:\mathbf{R}_{0}^{d}\rightarrow \mathbf{R}$ satisfies the
inequalities 
\begin{equation*}
|F(\xi )|\leq C|\xi |^{l},|\nabla F(\xi )|\leq C|\xi |^{l-1},\xi \in \mathbf{%
R}_{0}^{d}.
\end{equation*}%
Then%
\begin{equation*}
|\partial ^{{\scriptsize \delta }}F(\xi )|\leq C|\xi |^{l-{\scriptsize %
\delta }},\xi \in \mathbf{R}_{0}^{d}.
\end{equation*}
\end{lemma}

\begin{proof}
For any $\xi \in \mathbf{R}_{0}^{d},$%
\begin{eqnarray*}
|\partial ^{{\scriptsize \delta }}F(\xi )| &=& C \bigg\vert \int [F(\xi
+y)-F(\xi )]\frac{dy}{|y|^{d+{\scriptsize \delta }}}\bigg\vert \\
&\leq & C \int_{|y|>\frac{1}{2}|\xi |}[|F(\xi +y)|+|F(\xi )|]\frac{dy}{%
|y|^{d+{\scriptsize \delta }}} \\
&&+C\int_{|y|\leq \frac{1}{2}|\xi |}\int_{0}^{1}|\nabla F(\xi +sy)|\frac{%
dsdy }{|y|^{d+{\scriptsize \delta -1}}},
\end{eqnarray*}%
where the constant $C=C(\delta )$.

Changing the variable of integration, $y=|\xi |\bar{y},$ we have%
\begin{eqnarray*}
\int_{|y|>\frac{1}{2}|\xi |}|F(\xi +y)|\frac{dy}{|y|^{d+{\scriptsize \delta }%
}} &\leq &C\int |\xi +y|^{l}\frac{dy}{|y|^{d+{\scriptsize \delta }}} \\
&=&C|\xi |^{l-{\scriptsize \delta }}\int_{|\bar{y}|\geq \frac{1}{2}}~|\frac{%
\xi }{|\xi |}+\bar{y}|^{l}\frac{d\bar{y}}{|\bar{y}|^{d+{\scriptsize \delta }}%
} \\
&\leq &C|\xi |^{l-{\scriptsize \delta }}\sup_{|w|=1}\int_{|\bar{y}|\geq 
\frac{1}{2}}~|w+\bar{y}|^{l}\frac{d\bar{y}}{|\bar{y}|^{d+{\scriptsize \delta 
}}}.
\end{eqnarray*}%
Obviously,%
\begin{eqnarray*}
\int_{|y|\geq \frac{1}{2}|\xi |}|F(\xi )|\frac{dy}{|y|^{d+{\scriptsize %
\delta }}} &\leq &C|\xi |^{l}\int_{|y|\geq \frac{1}{2}|{\scriptsize \xi |}}%
\frac{dy}{|y|^{d+{\scriptsize \delta }}} \leq C|\xi |^{l-{\scriptsize \delta 
}}.
\end{eqnarray*}%
If $|y|\leq \frac{1}{2}|\xi |,s\in (0,1)$, then $|\xi +sy|\geq |\xi
|-s|y|\geq \frac{1}{2}|\xi |$ and 
\begin{eqnarray*}
\int_{|y|\leq \frac{1}{2}|{\scriptsize \xi |}}\int_{0}^{1}|\nabla F(\xi +sy)|%
\frac{dsdy}{|y|^{d+{\scriptsize \delta -1}}} &\leq &C\int_{|y|\leq \frac{1}{2%
}|{\scriptsize \xi |}} \int_{0}^{1}|\xi +sy|^{l-1}\frac{dsdy}{|y|^{d+%
{\scriptsize \delta -1}}} \\
&\leq &C\int_{|y|\leq \frac{1}{2}|{\scriptsize \xi |}}|\xi |^{l-1}\frac{dy}{%
|y|^{d+{\scriptsize \delta -1}}}\leq C|\xi |^{l-{\scriptsize \delta }}.
\end{eqnarray*}
\end{proof}

We will need some facts about maximal and sharp functions as well (see \cite%
{stein2}).

For each $(s,z)\in \mathbf{R}^{d+1}$ and $\delta >0$ we consider a family of
open sets $B(s,z;\delta )$ of the form%
\begin{equation*}
B(s,z;\delta )=(s-\delta ^{\alpha },s+\delta ^{\alpha })\times (z_{1}-\delta
,z_{1}+\delta )\times \ldots \times (z_{d}-\delta ,z_{d}+\delta ).
\end{equation*}%
Let $\mathbb{Q}_{\delta }$ be the family of all $B(s,z;\delta ),(s,z)\in 
\mathbf{R}^{d+1},$ and $\mathbb{Q=\cup }_{\delta >0}\mathbb{Q}_{\delta }$.
The collection $\mathbb{Q}$ satisfies the basic assumptions in \cite{stein2}
(see I.2.3 in \cite{stein2}).

Let $h\in L_{1}(\mathbf{R}^{d+1})$. For the rectangle $B\in \mathbb{Q}$ we
set%
\begin{eqnarray*}
h_{B} &=&\frac{1}{\text{mes}\,B}\int_{B}h(s,y)dsdy, \\
h_{B}^{\#} &=&\frac{1}{\text{mes}\,B}\int_{B}|h(s,y)-h_{B}|dsdy.
\end{eqnarray*}%
Let%
\begin{eqnarray*}
{M}h(t,x) &=&\sup_{\delta >0}\frac{1}{\text{mes}\,B(t,x;\delta )}%
\int_{B(t,x;\delta )}|h(s,y)|dsdy, \\
h^{\#}(t,x) &=&\sup_{B\in \mathbb{Q},(t,x)\in B}h_{B}^{\#},(t,x)\in \mathbf{R%
}^{d+1}.
\end{eqnarray*}%
In the definition of $h^{\#}$ the supremum is taken over all $B\in \mathbb{Q}%
=\cup _{\delta >0}\mathbb{Q}_{\delta }$ such that $(t,x)\in B$. The
functions ${M}h$ and $h^{\#}$ are called the maximal and sharp functions of $%
h$.

By H\"{o}lder's inequality for $h\in L_{2}(\mathbf{R}^{d+1}),$%
\begin{equation}
\left( h_{B}^{\#}\right) ^{2}\leq \frac{1}{\text{mes}B}%
\int_{B}h^{2}(s,y)dsdy,  \label{2.2}
\end{equation}%
\begin{equation}
\left( h_{B}^{\#}\right) ^{2}\leq \frac{1}{(\text{mes}~B)^{2}}%
\int_{B}\int_{B}(h(s,y)-h(u,z))^{2}dudzdsdy.  \label{2.3}
\end{equation}

We will also use the maximal functions defined by 
\begin{equation*}
\mathcal{M}f(x)=\sup_{r>0}\frac{1}{\mbox{mes}\,B_{r}(0)}%
\int_{B_{r}(x)}|f(y)|dy,
\end{equation*}%
where $B_{r}(x)=\{y\in \mathbf{R}^{d}:|y-x|<r\}$.

As it is well known (\cite{stein2}, Theorem IV.2.2, ), for $h\in L_{p}(%
\mathbf{R}^{d+1})$, $p>1$, the following norms are equivalent: 
\begin{equation}
|h|_{p}\sim |Mh|_{p}\sim |h^{\#}|_{p}.  \label{eq19}
\end{equation}%
Also, for $h\in L_{p}(\mathbf{R}^{d})$, $p>1$, 
\begin{equation}
|h|_{p}\sim |\mathcal{M}h|_{p}.  \label{eq20}
\end{equation}

\begin{lemma}
\label{r2} Let $f\in C_{0}^{\infty }(\mathbf{R}^{d})$ and $v$ be a
continuously differentiable function on $\mathbf{R}^{d}$ such that $%
\lim_{|z|\rightarrow \infty }|v(z)|=0$. Let $R,R_{1}\geq 0,\ x,y\in \mathbf{R%
}^{d}$, $|x-y|\leq R_{1}$ and $f(z)=0$ if $|y-z|\leq R$.

Then 
\begin{equation*}
|(f\ast v)(y)|\leq C\bigl[\mathcal{M}f^{2}(x)\bigr]^{\frac{1}{2}%
}\int_{R}^{\infty }(R_{1}+\rho )^{d}\Phi (\rho )d\rho ,
\end{equation*}%
where the constant $C=C(d)$ and 
\begin{equation*}
\Phi (\rho )=\biggl(\int_{|w|=1}\bigl(\nabla v(\rho w),w\bigr)^{2}dw\biggr)^{%
\frac{1}{2}},
\end{equation*}%
where $dw$ is the counting measure on $\left\{ -1,1\right\} $ if $d=1,$ and $%
dw$ is the Lebesgue measure if $d\geq 2.$
\end{lemma}

\begin{proof}
Integrating by parts, we have 
\begin{eqnarray*}
\int f(y-z)v(z)dz &=&\int_{R}^{\infty }\int_{|w|=1}f(y-\rho w)v(\rho w)\rho
^{d-1}dwd\rho \\
&=&\int_{R}^{\infty }\int_{|w|=1}v(\rho w)\frac{d}{d\rho }\int_{R}^{%
{\scriptsize \rho }}f(y-rw)r^{d-1}drdwd\rho \\
&=&\int_{|w|=1}\bigg[v(\rho w)\int_{R}^{{\scriptsize \rho }}f(y-rw)r^{d-1}dr%
\bigg]\bigg\vert_{R}^{\infty }dw \\
&&\quad -\int_{R}^{\infty }\int_{|w|=1}\int_{R}^{{\scriptsize \rho }%
}f(y-rw)r^{d-1}dr\bigl(\nabla v(\rho w),w\bigr)dwd\rho \\
&=&-\int_{R}^{\infty }\int_{|w|=1}\int_{R}^{{\scriptsize \rho }%
}f(y-rw)r^{d-1}dr\bigl(\nabla v(\rho w),w\bigr)dwd\rho .
\end{eqnarray*}%
Therefore, by H\"{o}lder's inequality, 
\begin{eqnarray*}
|(f\ast v)(y)| &\leq &\int_{R}^{\infty }\bigg(\int_{R}^{{\scriptsize \rho }%
}\int_{|w|=1}f^{2}(y-rw)r^{d-1}dwdr\bigg)^{\frac{1}{2}}\biggl(\int_{R}^{%
{\scriptsize \rho }}r^{d-1}dr\biggr)^{\frac{1}{2}}\Phi (\rho )d\rho \\
&\leq &C\int_{R}^{\infty }\biggl(\int_{B_{{\scriptsize \rho }}(y)}f^{2}(z)dz%
\biggr)^{\frac{1}{2}}\rho ^{\frac{d}{2}}\Phi (\rho )d\rho \\
&\leq &C\int_{R}^{\infty }\biggl(\int_{B_{R_{1}+{\scriptsize \rho }%
}(x)}f^{2}(z)dz\biggr)^{\frac{1}{2}}\rho ^{\frac{d}{2}}\Phi (\rho )d\rho \\
&\leq &C\int_{R}^{\infty }(R_{1}+\rho )^{\frac{d}{2}}\rho ^{\frac{d}{2}}%
\biggl(\sup_{{\scriptsize \rho >0}}(R_{1}+\rho )^{-d}\int_{B_{R_{1}+%
{\scriptsize \rho }}(x)}f^{2}(z)dz\biggr)^{\frac{1}{2}}\Phi (\rho )d\rho \\
&\leq &C\bigl[\mathcal{M}f^{2}(x)\bigr]^{\frac{1}{2}}\int_{R}^{\infty
}(R_{1}+\rho )^{d}\Phi (\rho )d\rho .
\end{eqnarray*}
\end{proof}

\subsection{Proof of Proposition \protect\ref{main2}}

1$^{0}.$ Since $(I-\Delta )^{{\scriptsize \beta /2}}:H_{p}^{s}\rightarrow
H_{p}^{s-{\scriptsize \beta /2}},s\in \mathbf{R},$ is an isomorphism (see 
\cite{Ste71}), it is enough to prove the first inequality for $\beta =0$.
Also, it is enough to consider $g\in C_{0}^{\infty }(\mathbf{R}^{d+1},V),$
the space of smooth $V$ -valued functions on $\mathbf{R}^{d+1}$ with compact
support.

Let us introduce the function 
\begin{equation*}
\tilde{\psi}^{({\scriptsize \alpha )}}(t,\xi )=\psi ^{({\scriptsize \alpha )}%
}\biggl(t,\frac{\xi }{|\xi |}\biggr),\quad \xi \in \mathbf{R}_{0}^{d}=\{\xi
\in \mathbf{R}^{d}:\xi \neq 0\}.
\end{equation*}%
Obviously, if $\alpha \neq 1,$ 
\begin{equation}
\psi ^{({\scriptsize \alpha )}}(t,\xi )=|\xi |^{{\scriptsize \alpha }}\tilde{%
\psi}^{({\scriptsize \alpha )}}(t,\xi ).  \label{du}
\end{equation}%
Since 
\begin{eqnarray*}
(w,\xi )\ln |(w,\xi )| &=&|\xi |(w,\frac{\xi }{|\xi |})\ln [|(w,\frac{\xi }{%
|\xi |})|\xi |] \\
&=&|\xi |(w,\frac{\xi }{|\xi |})\ln |(w,\frac{\xi }{|\xi |})+|\xi |(w,\frac{%
\xi }{|\xi |})\ln |\xi |
\end{eqnarray*}%
and $\int_{|w|=1}wm^{(1)}(t,w)dw=0$, the equality (\ref{du}) holds for $%
\alpha =1$ as well. By Assumption \textbf{B}, 
\begin{equation*}
\mbox{Re}\,\tilde{\psi}^{({\scriptsize \alpha )}}(t,\xi )\leq -\mu <0,\quad
t\in \mathbf{R},\ \xi \in \mathbf{R}_{0}^{d}.
\end{equation*}

Let $p=2$ and $g\in H^0_2(E_{a,b},V)$. Then, by Parseval's equality, 
\begin{eqnarray}
\vert \partial^{{\scriptsize \alpha /2}}\mathcal{I}g\vert^2_{H^0_2(\tilde
E_{a,b},V)} &=& \int_a^b\int_a^t\int \vert \partial^{{\scriptsize \alpha /2}}%
\mathcal{I}g(s,t,x)\vert^2 _V dxdsdt  \notag \\
&=&\int_a^b\int_a^t\int \big\vert \vert\xi\vert^{{\scriptsize \alpha /2}} 
\text{e}^{\vert {\scriptsize \xi\vert^{\alpha}\int_s^t\tilde{\psi}^{(\alpha
)}(r,\xi)dr}} \mathcal{F}g(s,\xi) \big\vert^2_V d\xi dsdt  \notag \\
&\leq&\int_a^b\int_a^t\int \vert\xi\vert^{{\scriptsize \alpha}} \text{e}^{-2%
{\scriptsize \mu\vert \xi\vert^{\alpha}(t-s)}} \vert \mathcal{F}g(s,\xi)
\vert^2_V d\xi dsdt  \notag \\
&=&\int\int_a^b\int_s^b \vert\xi\vert^{{\scriptsize \alpha}} \text{e}^{-2%
{\scriptsize \mu\vert \xi\vert^{\alpha}(t-s)}} \vert \mathcal{F}g(s,\xi)
\vert^2_V dtdsd\xi  \notag \\
&\leq&(2\mu)^{-1} \int_a^b\int \vert \mathcal{F}g(s,\xi) \vert^2_V d\xi ds 
\notag \\
&=& (2\mu)^{-1} \vert g\vert^2_{H^0_2(E_{a,b},V)} .  \label{eq23}
\end{eqnarray}

2$^{0}$. Let $p>2$. We extend the functions $g\in H_{p}^{0}(E_{a,b},V)$ by
zero outside the interval $[a,b]$ if necessary. Obviously, the extended
functions belong to $H_{p}^{0}(E,V)$, where $E=E_{-\infty ,\infty }=\mathbf{R%
}^{d+1}$.

For $g\in H_{p}^{0}(E,V)$ we denote%
\begin{eqnarray*}
Gg(s,y) &=&\left\{ \int_{-\infty }^{s}\left\vert \int \partial ^{%
{\scriptsize \alpha /2}}G_{u,s}(y-y^{\prime })g(u,y^{\prime })dy^{\prime
}\right\vert _{V}^{2}du\right\} ^{1/2} \\
&=&\left\{ \int_{-\infty }^{s}\left\vert \int G_{u,s}(y-y^{\prime })\partial
^{{\scriptsize \alpha /2}}g(u,y^{\prime })dy^{\prime }\right\vert
_{V}^{2}du\right\} ^{1/2}.
\end{eqnarray*}%
Note that by triangle inequality in $L_{2}((-\infty ,s],V)$ we have for $%
g_{1},g_{2},\in H_{p}^{0}(E,V),$ 
\begin{eqnarray}
G(g_{1}+g_{2})(s,y) &\leq &Gg_{1}(s,y)+Gg_{2}(s,y),  \label{maf2} \\
|G(g_{1}+g_{2})(s,y)-Gg_{1}(s,y)| &\leq &Gg_{2}(s,y).  \notag
\end{eqnarray}%
According to (\ref{eq19}) and (\ref{eq20}) it is enough to prove that there
is a constant $C$ such that for all $g\in H_{p}^{0}(E,V),\ (t,x)\in \mathbf{R%
}^{d+1}$ 
\begin{equation}
\left( Gg\right) ^{\#}(t,x)\leq C(\mathcal{M}_{t}\mathcal{M}%
_{x}|g|_{V}^{2}(t,x))^{1/2},  \label{2.5}
\end{equation}%
where $\mathcal{M}_{t}$ and $\mathcal{M}_{x}$ denote the maximal functions
defined using the balls in $\mathbf{R}$ and $\mathbf{R}^{d}$ and 
\begin{equation*}
\left( Gg\right) ^{\#}(t,x)=\sup_{B\in \mathbb{Q},(t,x)\in B}\frac{1}{\text{%
mes(}B)}\int_{B}|Gg(s,y)-(Gg)_{B}|dsdy.
\end{equation*}

Since $B\in \mathbb{Q}$ is of the form%
\begin{eqnarray*}
B &=&(s_{0}-\delta ^{\alpha },s_{0}+\delta ^{\alpha })\times (z_{1}-\delta
,z_{1}+\delta )\times \ldots \times (z_{d}-\delta ,z_{d}+\delta )\} \\
&=&(\tilde{s}_{0},z)+\tilde{B}(0,0;\delta ),
\end{eqnarray*}%
with $\tilde{s}_{0}=s_{0}+\delta ^{\alpha }$,$\tilde{B}(0,0;\delta
)=(-2\delta ^{\alpha },0)\times (-\delta ,\delta )^{d},$ it is
straightforward to verify that%
\begin{eqnarray*}
&&\frac{1}{\text{mes(}B)}\int_{B}|Gg(s,y)-(Gg)_{B}|dsdy \\
&=&\frac{1}{\text{mes(}Q_{0})}\int_{Q_{0}}|Gg(\tilde{s}_{0}+\delta ^{\alpha
}s,z+\delta y)-(Gg(\tilde{s}_{0}+\delta ^{\alpha }\cdot ,z+\delta \cdot
))_{Q_{0}}|dsdy,
\end{eqnarray*}%
where $Q_{0}=\tilde{B}(0,0;1)$.

Changing the variable of integration, $u=\tilde{s}_{0}+\delta ^{\alpha
}s,y^{\prime }=z+\delta y$, we see that 
\begin{eqnarray*}
&&Gg(\tilde{s}_{0}+\delta ^{\alpha }t,z+\delta x) \\
&=&\left\{ \int_{-\infty }^{\tilde{s}_{0}+\delta ^{\alpha }t}\left\vert \int
\partial ^{{\scriptsize \alpha /2}}G_{u,\tilde{s}_{0}+\delta ^{\alpha
}t}(z+\delta x-y^{\prime })g(u,y^{\prime })dy^{\prime }\right\vert
_{V}^{2}du\right\} ^{1/2} \\
&=&\delta ^{\frac{\alpha }{2}+d}\left\{ \int_{-\infty }^{t}\left\vert \int
\partial ^{{\scriptsize \alpha /2}}G_{\tilde{s}_{0}+\delta ^{\alpha }s,%
\tilde{s}_{0}+\delta ^{\alpha }t}(\delta (x-y))g(\tilde{s}_{0}+\delta
^{\alpha }s,z+\delta y)dy\right\vert _{V}^{2}ds\right\} ^{1/2} \\
&=&\left\{ \int_{-\infty }^{t}\left\vert \int \partial ^{{\scriptsize \alpha
/2}}G_{s,t}^{\tilde{s}_{0},\delta }(x-y)g(\tilde{s}_{0}+\delta ^{\alpha
}s,z+\delta y)dy\right\vert _{V}^{2}ds\right\} ^{1/2},
\end{eqnarray*}%
where%
\begin{equation*}
G_{s,t}^{s_{0},\delta }(x)=\mathcal{F}^{-1}\left( \exp \left\{
\int_{s}^{t}\psi ^{({\scriptsize \alpha )}}(s_{0}+\delta ^{\alpha }r,\xi
)dr\right\} \right) (x)
\end{equation*}%
with%
\begin{eqnarray*}
{}\psi ^{({\scriptsize \alpha )}}(s_{0}+\delta ^{\alpha }t,\xi )
&=&i(b(s_{0}+\delta ^{\alpha }t),\xi )1_{{\scriptsize \alpha =1}%
}-\sum_{i,j=1}^{d}B^{ij}(s_{0}+\delta ^{\alpha }t)\xi _{i}\xi _{j}1_{%
{\scriptsize \alpha =2}} \\
&&-C\int_{S^{d-1}}|(w,\xi )|^{{\scriptsize \alpha }}\Big[1-i\Big(\tan \frac{%
\alpha \pi }{2}\mbox{sgn}(w,\xi )1_{{\scriptsize \alpha \neq 1}} \\
&&-\frac{2}{\pi }\mbox{sgn}(w,\xi )\ln |(w,\xi )|1_{{\scriptsize \alpha =1}}%
\Big)\Big]m^{(\alpha )}(s_{0}+\delta ^{\alpha }t,w)dw.
\end{eqnarray*}%
Note that for every $\tilde{s}_{0}\in \mathbf{R}^{d},\delta >0,$ the
coefficients $b(s_{0}+\delta ^{\alpha }t),B^{ij}(s_{0}+\delta ^{\alpha
}t),m^{(\alpha )}(s_{0}+\delta ^{\alpha }t,w),t\in \mathbf{R,}w\in S^{d-1},$
satisfy the assumptions \textbf{A,B} with the same constants $K$ and $\mu $.
Therefore for (\ref{2.5}) it is enough to show the inequality 
\begin{equation}
\left( Gg\right) _{Q_{0}}^{\#}\leq C\left( \mathcal{M}_{t}\mathcal{M}%
_{x}|g(t,x\right) |_{V}^{2})^{1/2},(t,x)\in Q_{0},  \label{2.6}
\end{equation}%
with 
\begin{equation*}
Q_{0}=\tilde{B}(0,0;1)=\left\{ (t,x)\in \lbrack -2,0]\times \lbrack
-1,1]^{d}\right\}
\end{equation*}

We consider the following three cases: \vskip6pt

(1) $g(t,x)=0,(t\,,x)\notin \lbrack -12,12]\times B_{3\sqrt{d}}(0);$

\vskip3pt (2) $g(t,x)=0,(t,x)\notin \lbrack -12,12]\times \mathbf{R}^{d};$

\vskip3pt (3) $g(t,x)=0,t\geq -8,x\in \mathbf{R}^{d}\,$.

For the estimates of the derivatives of $G_{u,s}(x)$ the following
representation is helpful. For $u<s,x\in \mathbf{R}^{d},j,k=1,\ldots ,d,$%
\begin{eqnarray}
\partial _{j}\partial _{k}\partial ^{{\scriptsize \alpha /2}}G_{u,s}(x)
&=&(s-u)^{-\frac{d}{{\scriptsize \alpha }}-\frac{1}{2}-\frac{2}{{\scriptsize %
\alpha }}}F_{u,s}^{j,k}\left( (s-u)^{-\frac{1}{{\scriptsize \alpha }}%
}x\right) ,  \label{maf3} \\
\partial _{j}\partial ^{{\scriptsize \alpha /2}}G_{u,s}(x) &=&(s-u)^{-\frac{d%
}{{\scriptsize \alpha }}-\frac{1}{2}-\frac{1}{{\scriptsize \alpha }}%
}F_{u,s}^{j}\left( (s-u)^{-\frac{1}{{\scriptsize \alpha }}}x\right) ,  \notag
\\
\partial _{s}\partial ^{\frac{{\scriptsize \alpha }}{2}}G_{u,s}(x)
&=&(s-u)^{-\frac{d}{{\scriptsize \alpha }}-\frac{3}{2}}\bar{F}_{u,s}((s-u)^{-%
\frac{1}{{\scriptsize \alpha }}}x),  \notag \\
\partial _{j}\partial _{s}\partial ^{\frac{{\scriptsize \alpha }}{2}%
}G_{u,s}(x) &=&(s-u)^{-\frac{d}{{\scriptsize \alpha }}-\frac{3}{2}-\frac{1}{%
{\scriptsize \alpha }}}\bar{F}_{u,s}^{j}((s-u)^{-\frac{1}{{\scriptsize %
\alpha }}}x),  \notag
\end{eqnarray}%
with 
\begin{eqnarray*}
F_{u,s}^{j,k} &=&\mathcal{F}^{-1}\left\{ -\xi _{j}\xi _{k}|\xi |^{\frac{%
{\scriptsize \alpha }}{2}}\exp \left\{ -|\xi |^{{\scriptsize \alpha }}\frac{1%
}{(s-u)}\int_{u}^{s}\tilde{\psi}^{({\scriptsize \alpha )}}(r,\xi )dr\right\}
\right\} , \\
F_{u,s}^{j} &=&\mathcal{F}^{-1}\left\{ i\xi _{k}|\xi |^{\frac{{\scriptsize %
\alpha }}{2}}\exp \left\{ -|\xi |^{{\scriptsize \alpha }}\frac{1}{(s-u)}%
\int_{u}^{s}\tilde{\psi}^{({\scriptsize \alpha )}}(r,\xi )dr\right\}
\right\} , \\
\bar{F}_{u,s} &=&\mathcal{F}^{-1}\left\{ -|\xi |^{\frac{3}{2}{\scriptsize %
\alpha }}\tilde{\psi}^{({\scriptsize \alpha )}}(s,\xi )\exp \left\{ -|\xi |^{%
{\scriptsize \alpha }}\frac{1}{(s-u)}\int_{u}^{s}\tilde{\psi}^{({\scriptsize %
\alpha )}}(r,\xi )dr\right\} \right\} , \\
\bar{F}_{u,s}^{j} &=&\mathcal{F}^{-1}\left\{ -i\xi _{j}|\xi |^{\frac{3}{2}%
{\scriptsize \alpha }}\tilde{\psi}^{({\scriptsize \alpha )}}(s,\xi )\exp
\left\{ -|\xi |^{{\scriptsize \alpha }}\frac{1}{(s-u)}\int_{u}^{s}\tilde{\psi%
}^{({\scriptsize \alpha )}}(r,\xi )dr\right\} \right\} .
\end{eqnarray*}%
By definition of the inverse Fourier transform, all functions $%
F_{u,s}^{j,k},F_{u,s}^{j},\bar{F}_{u,s},\bar{F}_{u,s}^{j}$ are uniformly
bounded.

\vskip6pt 3$^{0}$. First, we prove that in the case (1) 
\begin{equation}
\int_{Q_{0}}(Gg)(s,y)^{2}dsdy\leq C\mathcal{M}_{t}\mathcal{M}%
_{x}|g(t,x)|_{V}^{2}  \label{2.7}
\end{equation}%
for all $(t,x)\in Q_{0}.$

Repeating the proof of (\ref{eq23}), we have 
\begin{eqnarray*}
\int_{Q_{0}}(Gg)^{2}(s,y)dsdy &\leq &\int_{-\infty }^{\infty }\int_{-\infty
}^{s}\int |\xi |^{{\scriptsize \alpha }}e^{-2{\scriptsize \mu |\xi |^{\alpha
}(s-u)}}|\mathcal{F}g(u,\xi )|_{V}^{2}d\xi duds \\
&\leq &(2\mu )^{-1}\int_{-\infty }^{\infty }\int |g(u,y)|_{V}^{2}dudy \\
&=&(2\mu )^{-1}\int_{-12}^{12}\int_{B_{2\sqrt{d}}(0)}|g(u,y)|_{V}^{2}dydu.
\end{eqnarray*}%
Now for every $(t,x)\in Q_{0},$%
\begin{eqnarray*}
&&\int_{-12}^{12}\int_{B_{3\sqrt{d}}(0)}|g(u,y)|_{V}^{2}dydu \\
&&\quad \leq \text{mes}\,(B_{5\sqrt{d}}(0))\int_{-12}^{12}\frac{1}{\text{mes}%
\,(B_{5\sqrt{d}}(x))}\int_{B_{5\sqrt{d}}(x)}|g(u,y)|_{V}^{2}dydu \\
&&\quad \leq \text{mes}\,B_{5\sqrt{d}}(0)\int_{-12}^{12}\mathcal{M}%
_{x}|g(u,x)|_{V}^{2}du \\
&&\quad \leq C\mathcal{M}_{t}\mathcal{M}_{x}|g(t,x)|_{V}^{2}
\end{eqnarray*}%
and (\ref{2.7}) is proven.

4$^{0}$. Now we prove that (\ref{2.7}) holds in the case (2) as well. Since (%
\ref{2.7}) holds for $g(t,z)=0,(t\,,z)\notin \lbrack -12,12]\times B_{3\sqrt{%
d}}(0)$, it is enough to consider $g(t,z)$ such that $g(t,z)=0$ if $|t|>12$
or $|z|\leq 2\sqrt{d}$. By Minkowski's inequality,%
\begin{eqnarray*}
(Gg)^{2}(s,y) &=&\int_{-\infty }^{s}\left\vert \int \partial ^{{\scriptsize %
\alpha /2}}G_{u,s}(y-y^{\prime })g(u,y^{\prime })dy^{\prime }\right\vert
_{V}^{2}du \\
&\leq &\int_{-12}^{s}\left( \int |\partial ^{{\scriptsize \alpha /2}%
}G_{u,s}(y-y^{\prime })|~|g(u,y^{\prime })|_{V}dy^{\prime }\right) ^{2}du.
\end{eqnarray*}%
According to Lemma \ref{r2} (in our case $R=\sqrt{d},R_{1}=2\sqrt{d}),$ 
\begin{eqnarray*}
&&\left( \int |\partial ^{{\scriptsize \alpha /2}}G_{u,s}(y-y^{\prime
})|~|g(u,y^{\prime })|_{V}dy^{\prime }\right) ^{2}\leq C\mathcal{M}%
_{x}|g(u,x)|_{V}^{2}\times \\
&&\qquad \times \left( \int_{\sqrt{d}}^{\infty }(2\sqrt{d}+\rho )^{d}\left(
\int_{|w|=1}\sum_{j=1}^{d}|\partial _{j}\partial ^{{\scriptsize \alpha /2}%
}G_{u,s}(\rho w)|^{2}dw\right) ^{1/2}d\rho \right) ^{2} \\
&&\quad \leq C\mathcal{M}_{x}|g(u,x)|_{V}^{2}\kappa (u,s),
\end{eqnarray*}%
where%
\begin{equation*}
\kappa (u,s)=\left( \int_{1}^{\infty }\rho ^{d}\left(
\int_{|w|=1}\sum_{j=1}^{d}|\partial _{j}\partial ^{{\scriptsize \alpha /2}%
}G_{u,s}(\rho w)|^{2}dw\right) ^{1/2}d\rho \right) ^{2}.
\end{equation*}%
By (\ref{maf3}),%
\begin{equation*}
\kappa (u,s)=(s-u)^{-\frac{2d}{{\scriptsize \alpha }}-1-\frac{2}{%
{\scriptsize \alpha }}}\left( \int_{1}^{\infty }\rho ^{d}\left(
\int_{|w|=1}\sum_{j=1}^{d}|F_{u,s}^{j}(\rho (s-u)^{-\frac{1}{{\scriptsize %
\alpha }}}w)|^{2}dw\right) ^{1/2}d\rho \right) ^{2}.
\end{equation*}%
Changing the variable of integration, $\rho (s-u)^{-\frac{1}{{\scriptsize %
\alpha }}}=r,$ and using H\"{o}lder's inequality, we get%
\begin{eqnarray*}
{}\kappa (u,s) &=&(s-u)^{-1}\left( \int_{(s-u)^{-\frac{1}{{\scriptsize %
\alpha }}}}^{\infty }r^{d}\left(
\int_{|w|=1}\sum_{j=1}^{d}[F_{u,s}^{j}(rw)]^{2}dw\right) ^{1/2}dr\right) ^{2}
\\
&\leq &(s-u)^{-1}\int_{(s-u)^{-\frac{1}{{\scriptsize \alpha }}}}^{\infty
}r^{-1-\frac{{\scriptsize \alpha }}{2}}dr\int_{0}^{\infty }r^{2d+1+\frac{%
{\scriptsize \alpha }}{2}}\int_{|w|=1}\sum_{j=1}^{d}[F_{u,s}^{j}(rw)]^{2}dwdr
\\
&\leq &C(s-u)^{-\frac{1}{2}}\int \sum_{j=1}^{d}\Bigl[|x|^{\frac{d}{2}+1+%
\frac{{\scriptsize \alpha }}{4}}F_{u,s}^{j}(x)\Bigr]^{2}dx.
\end{eqnarray*}%
Hence, by Parseval's equality,%
\begin{equation*}
{}\kappa (s,u)\leq C(s-u)^{-\frac{1}{2}}\int \sum_{j=1}^{d}\big\vert\partial
^{\frac{d}{2}+1+\frac{{\scriptsize \alpha }}{4}}\mathcal{F}F_{u,s}^{j}(\xi )%
\big\vert^{2}d\xi .
\end{equation*}%
Due to our assumptions \textbf{A}, \textbf{B} and Lemma \ref{auxl2}, the
last integral is finite. Therefore%
\begin{eqnarray}
\int_{Q_{0}}(Gg)^{2}dsdy &\leq &C\int_{-2}^{0}\int_{-12}^{s}(s-u)^{-\frac{1}{%
2}}\mathcal{M}_{x}|g(u,x)|_{V}^{2}duds  \notag \\
&=&C\left(
\int_{-12}^{-2}\int_{-2}^{0}l(s,u,x)~dsdu+\int_{-2}^{0}%
\int_{u}^{0}l(s,u,x)~dsdu\right)  \notag \\
&\leq &C\int_{-12}^{0}\mathcal{M}_{x}|g(u,x)|_{V}^{2}du\leq C\mathcal{M}_{t}%
\mathcal{M}_{x}|g(t,x)|_{V}^{2}  \label{eq25}
\end{eqnarray}%
for all $(t,x)\in Q_{0}$ with%
\begin{equation*}
l(s,u,x)=(s-u)^{-\frac{1}{2}}\mathcal{M}_{x}|g(u,x)|_{V}^{2}.
\end{equation*}%
$.$

5$^{0}.$ We will show that in the case (3)%
\begin{equation}
\int_{Q_{0}}|Gg(s,y)-Gg(t^{\prime },x^{\prime })|^{2}dsdy\leq C\mathcal{M}%
_{t}\mathcal{M}_{x}|g|_{V}^{2}(t,x)  \label{2.10}
\end{equation}%
with all $(t,x),(t^{\prime },x^{\prime })\in Q_{0}$. We estimate the
Lipschitz constant of $Gg$ in $t$ and $x$. Obviously, for each $(s,y),\
(t^{\prime },x^{\prime })\in Q_{0}$ 
\begin{equation}
\left\vert Gg(s,y)-Gg(t^{\prime },x^{\prime })\right\vert \leq C\left(
\sup_{(s,y)\in Q_{0}}|\nabla Gg(s,y)|+\sup_{(s,y)\in Q_{0}}|\partial
_{s}Gg(s,y)|\right) .  \label{maf4}
\end{equation}

First we estimate $|\nabla Gg(s,y)|$ in (\ref{maf4}). Let $\varphi \in
C_{0}^{\infty }(\mathbf{R}^{d}),0\leq \varphi \leq 1,\varphi (x)=1$ if $%
|x|\leq 2\sqrt{d}$, $\varphi (x)=0$ if $|x|<3\sqrt{d}$, and%
\begin{eqnarray*}
g_{2}(u,y^{\prime }) &=&g(u,y^{\prime })\varphi (y^{\prime }), \\
g_{1}(u,y^{\prime }) &=&g(u,y^{\prime })\left( 1-\varphi (y^{\prime
})\right) ,(u,y^{\prime })\in \mathbf{R}^{d+1}\text{.}
\end{eqnarray*}%
Since $g(u,y^{\prime })=0$ if $u\geq -8,$ applying H\"{o}lder's and
Minkowski's inequalities, we derive for $s\in \lbrack -2,0],|y|\leq 1,$ 
\begin{eqnarray*}
|\nabla Gg(s,y)|^{2} &\leq &\int_{-\infty }^{-8}\left\vert \int \nabla
\partial ^{{\scriptsize \alpha /2}}G_{u,s}(y-y^{\prime })g(u,y^{\prime
})dy^{\prime }\right\vert _{V}^{2}du \\
&\leq &2\int_{-\infty }^{-8}\left( \int_{|y^{\prime }|>2\sqrt{d}}|\nabla
\partial ^{{\scriptsize \alpha /2}}G_{u,s}(y-y^{\prime
})|~|g_{1}(u,y^{\prime })|_{V}dy^{\prime }\right) ^{2}du \\
&&+2\int_{-\infty }^{-8}\left( \int_{|y^{\prime }|\leq 3\sqrt{d}}|\nabla
\partial ^{{\scriptsize \alpha /2}}G_{u,s}(y-y^{\prime
})|~|g_{2}(u,y^{\prime })|_{V}dy^{\prime }\right) ^{2}du \\
&=&2(A_{1}(s,y)+A_{2}(s,y)).
\end{eqnarray*}%
For any $(t,x)\in Q_{0},$ according to (\ref{maf3}) and Lemma \ref{r1}
(applied for $d=1),$%
\begin{eqnarray*}
A_{2}(s,y) &\leq &\int_{-\infty }^{-8}\sup_{|z|\leq 4\sqrt{d}}|\nabla
\partial ^{{\scriptsize \alpha /2}}G_{u,s}(z)|^{2}(\int_{|y^{\prime }|\leq 3%
\sqrt{d}}|g(u,y^{\prime })|_{V}dy^{\prime })^{2}du \\
&\leq &C\int_{-\infty }^{-8}\sup_{|z|\leq 4\sqrt{d}}|\nabla \partial ^{%
{\scriptsize \alpha /2}}G_{u,s}(z)|^{2}(\frac{1}{\text{mes~}B_{4\sqrt{d}}(x)}%
\int_{|x-y^{\prime }|\leq 4\sqrt{d}}|g(u,y^{\prime })|_{V}^{2}dy^{\prime })
\\
&\leq &C\int_{-\infty }^{-8}(s-u)^{-\frac{2d}{{\scriptsize \alpha }}-1-\frac{%
2}{{\scriptsize \alpha }}}\mathcal{M}_{x}(|g|_{V}^{2})(u,x)du\leq C\mathcal{M%
}_{t}\mathcal{M}_{x}(|g|_{V}^{2})(t,x)
\end{eqnarray*}%
According to Lemma \ref{r2} (in our case $R=\sqrt{d}$ and $R_{1}=2\sqrt{d}$),%
\begin{eqnarray*}
&&\left( \int_{|y^{\prime }|\geq 2\sqrt{d}}|\nabla \partial ^{{\scriptsize %
\alpha /2}}G_{u,s}(y-y^{\prime })|~|g_{1}(u,y^{\prime })|_{V}dy^{\prime
}\right) ^{2}\leq C\mathcal{M}_{x}|g(u,x)|_{V}^{2}\times  \\
&&\qquad \times \bigg(\int_{\sqrt{d}}^{\infty }(2\sqrt{d}+\rho )^{d}\bigg(%
\int_{|w|=1}\sum_{j=1}^{d}|\partial _{j}\nabla \partial ^{{\scriptsize %
\alpha /2}}G_{u,s}(\rho w)|^{2}dw\bigg)^{1/2}d\rho \bigg )^{2} \\
&&\quad \leq C\mathcal{M}_{x}|g(u,x)|_{V}^{2}\tilde{\kappa}(u,s),
\end{eqnarray*}%
where%
\begin{equation*}
\tilde{\kappa}(u,s)=\bigg (\int_{1}^{\infty }\rho ^{d}\bigg (%
\int_{|w|=1}\sum_{j=1}^{d}|\partial _{j}\nabla \partial ^{{\scriptsize %
\alpha /2}}G_{u,s}(\rho w)|^{2}dw\bigg )^{1/2}d\rho \bigg )^{2}.
\end{equation*}

By (\ref{maf3}) and H\"{o}lder's inequality%
\begin{eqnarray*}
\tilde{\kappa}(u,s) &=&(s-u)^{-p}\bigg(\int_{1}^{\infty }\rho ^{d}\bigg(%
\int_{|w|=1}\sum_{j,k=1}^{d}\Big[F_{u,s}^{j,k}(\rho (s-u)^{-\frac{1}{%
{\scriptsize \alpha }}}w)\Big]^{2}dw\bigg)^{1/2}d\rho \bigg)^{2} \\
&\leq &(s-u)^{-p}\int_{1}^{\infty }\rho ^{-2}d\rho \int_{1}^{\infty }\rho
^{2d+2}\int_{|w|=1}\sum_{j,k=1}^{d}\Bigl[F_{u,s}^{j,k}\big(\rho (s-u)^{-%
\frac{1}{{\scriptsize \alpha }}}w\big)\Bigr]^{2}dwd\rho \\
&=&(s-u)^{-p}\int_{|x|\geq 1}|x|^{d+3}\sum_{j,k=1}^{d}\Bigl[F_{u,s}^{j,k}%
\big((s-u)^{-\frac{1}{{\scriptsize \alpha }}}x\big)\Bigr]^{2}dx
\end{eqnarray*}%
with $p=\frac{2d+4}{{\scriptsize \alpha }}+1$.

Changing the variable of integration, $y=(s-u)^{-\frac{1}{{\scriptsize %
\alpha }}}x$, we get by Parseval's equality%
\begin{eqnarray*}
\tilde{\kappa}(u,s) &\leq &(s-u)^{-1-\frac{1}{{\scriptsize \alpha }}}\int
|y|^{d+3}\sum_{j,k=1}^{d}\big[F_{u,s}^{j,k}(y)\big]^{2}dy \\
&=&(s-u)^{-1-\frac{1}{{\scriptsize \alpha }}}\int \sum_{j,k=1}^{d}\Big\vert%
\partial ^{\frac{d+3}{2}}\mathcal{F}F_{u,s}^{j,k}(\xi )\Big\vert^{2}d\xi .
\end{eqnarray*}%
Due to our assumptions and Lemma~\ref{auxl2}, the last integral is finite.
Hence, 
\begin{equation*}
\tilde{\kappa}(u,s)\leq C(s-u)^{-1-\frac{1}{{\scriptsize \alpha }}}
\end{equation*}%
and for $(s,y)\in Q_{0},$%
\begin{equation*}
A_{1}(s,y)\leq \int_{-\infty }^{-8}\mathcal{M}_{x}|g|_{V}^{2}(u,x)\tilde{%
\kappa}(u,s)du\leq C\int_{-\infty }^{-8}\mathcal{M}%
_{x}|g|_{V}^{2}(u,x)(s-u)^{-1-\frac{1}{{\scriptsize \alpha }}}du.
\end{equation*}

Therefore by Lemma \ref{r1} (in the case $d=1),$ for $(s,y)\in
Q_{0},(t,x)\in Q_{0},$ 
\begin{equation}
|\nabla Gg(s,y)|^{2}\leq A_{1}(s,y)+A_{2}(s,y)\leq C\mathcal{M}_{t}\mathcal{M%
}_{x}|g|_{V}^{2}(t,x).  \label{28}
\end{equation}

Now we estimate $|\partial _{s}Gg(s,y)|$. Applying H\"{o}lder's and
Minkowski's inequalities, we get for $(s,y)\in Q_{0},$ 
\begin{eqnarray*}
\lbrack \partial _{s}G(s,y)]^{2} &\leq &\int_{-\infty }^{-8}\left\vert \int
\partial _{s}\partial ^{{\scriptsize \alpha /2}}G_{u,s}(y-y^{\prime
})g(u,y^{\prime })dy^{\prime }\right\vert _{V}^{2}du \\
&\leq &2\int_{-\infty }^{-8}\bigg(\int_{|y^{\prime }|>2\sqrt{d}}|\partial
_{s}\partial ^{{\scriptsize \alpha /2}}G_{u,s}(y-y^{\prime
})|\,|g_{1}(u,y^{\prime })|_{V}dy^{\prime }\bigg)^{2}du \\
&&+\int_{-\infty }^{-8}\bigg(\int_{|y^{\prime }|\leq 3\sqrt{d}}|\partial
_{s}\partial ^{{\scriptsize \alpha /2}}G_{u,s}(y-y^{\prime
})|\,|g_{2}(u,y^{\prime })|_{V}dy^{\prime }\bigg)^{2}du \\
&=&2B_{1}(s,y)+2B_{2}(s,y).
\end{eqnarray*}%
According to Lemma~\ref{r2},%
\begin{eqnarray*}
&&\bigg(\int_{|y^{\prime }|>2\sqrt{d}}|\partial _{s}\partial ^{{\scriptsize %
\alpha /2}}G_{u,s}(y-y^{\prime })|\,|g_{1}(u,y^{\prime })|_{V}dy^{\prime }%
\bigg)^{2}\leq C\mathcal{M}_{x}|g|_{V}^{2}(u,x)\times \\
&&\qquad \times \biggl(\int_{\sqrt{d}}^{\infty }(2\sqrt{d}+\rho )^{d}\biggl(%
\int_{|w|=1}\sum_{j=1}^{d}\bigl[\partial _{j}\partial _{s}\partial ^{%
{\scriptsize \alpha /2}}G_{u,s}(\rho w)\bigr]^{2}dw\biggr)^{\frac{1}{2}%
}d\rho \biggr)^{2} \\
&&\quad \leq C\bar{\kappa}(u,s)\mathcal{M}_{x}|g|_{V}^{2}(u,x),
\end{eqnarray*}%
where%
\begin{equation*}
\bar{\kappa}(u,s)=\bigg(\int_{1}^{\infty }\rho ^{d}\bigg(\int_{|w|=1}%
\sum_{j=1}^{d}\big[\partial _{j}\partial _{s}\partial ^{{\scriptsize \alpha
/2}}G_{u,s}(\rho w)\big]^{2}dw\bigg)^{\frac{1}{2}}d\rho \bigg)^{2}.
\end{equation*}%
According to (\ref{maf3}), we have by H\"{o}lder's inequality 
\begin{eqnarray*}
\bar{\kappa}(u,s) &=&(s-u)^{-p}\bigg(\int_{1}^{\infty }\rho ^{d}\bigg(%
\int_{|w|=1}\sum_{j=1}^{d}\big[\bar{F}_{u,s}^{j}(\rho (s-u)^{-\frac{1}{%
{\scriptsize \alpha }}}w)\big]^{2}dw\bigg)^{\frac{1}{2}}d\rho \bigg)^{2} \\
&\leq &(s-u)^{-p}\int_{1}^{\infty }\rho ^{-1-{\scriptsize \alpha }}d\rho
\int_{1}^{\infty }\rho ^{2d+1+{\scriptsize \alpha }}\int_{|w|=1}%
\sum_{j=1}^{d}\big[\bar{F}_{u,s}^{j}(\rho (s-u)^{-\frac{1}{{\scriptsize %
\alpha }}}w)\big]^{2}dwd\rho \\
&\leq &C(s-u)^{-p}\int_{|x|\geq 1}|x|^{d+2+{\scriptsize \alpha }%
}\sum_{j=1}^{d}\big[\bar{F}_{u,s}^{j}((s-u)^{-\frac{1}{{\scriptsize \alpha }}%
}x)\big]^{2}dx,
\end{eqnarray*}%
where $p=\frac{2d}{{\scriptsize \alpha }}+3+\frac{2}{{\scriptsize \alpha }}$%
. Changing the variable of integration, $y=(s-u)^{-\frac{1}{{\scriptsize %
\alpha }}}x$, we get by Parseval's equality 
\begin{eqnarray*}
\bar{\kappa}(u,s) &\leq &C(s-u)^{-2}\int \sum_{j=1}^{d}\big[|y|^{\frac{d}{2}%
+1+\frac{{\scriptsize \alpha }}{2}}\bar{F}_{u,s}^{j}(y)\big]^{2}dy \\
&\leq &C(s-u)^{-2}\int \sum_{j=1}^{d}\big\vert\partial ^{\frac{d}{2}+1+\frac{%
{\scriptsize \alpha }}{2}}\mathcal{F}\bar{F}_{u,s}^{j}(\xi )\big\vert%
^{2}d\xi .
\end{eqnarray*}%
Due to our assumptions and Lemma~\ref{auxl2}, the last integral is finite.
Hence, 
\begin{equation*}
\bar{\kappa}(u,s)\leq C(s-u)^{-2}
\end{equation*}%
and, by Lemma \ref{r1} $(d=1)$ it follows for $(s,y),(t,x)\in Q_{0}$, 
\begin{equation}
B_{1}(s,y)\leq C\int_{-\infty }^{-8}(s-u)^{-2}\mathcal{M}%
_{x}|g|_{V}^{2}(u,x)du\leq C\mathcal{M}_{t}\mathcal{M}_{x}|g|_{V}^{2}(t,x).
\label{eq29}
\end{equation}%
For any $(s,y),(t,x)\in Q_{0},$ according to (\ref{maf3}) and Lemma \ref{r1}
($d=1),$%
\begin{eqnarray*}
B_{2}(s,y) &\leq &\int_{-\infty }^{-8}\sup_{|z|\leq 4\sqrt{d}}|\partial
_{s}\partial ^{{\scriptsize \alpha /2}}G_{u,s}(z)|^{2}(\int_{|y^{\prime
}|\leq 3\sqrt{d}}|g(u,y^{\prime })|_{V}dy^{\prime })^{2}du \\
&\leq &C\int_{-\infty }^{-8}\sup_{|z|\leq 4\sqrt{d}}|\partial _{s}\partial ^{%
{\scriptsize \alpha /2}}G_{u,s}(z)|^{2}(\frac{1}{\text{mes~}B_{4\sqrt{d}}(x)}%
\int_{|x-y^{\prime }|\leq 4\sqrt{d}}|g(u,y^{\prime })|_{V}^{2}dy^{\prime })
\\
&\leq &C\int_{-\infty }^{-8}(s-u)^{-\frac{2d}{{\scriptsize \alpha }}-3}%
\mathcal{M}_{x}|g|_{V}^{2}(u,x)du\leq C\mathcal{M}_{t}\mathcal{M}%
_{x}(|g|_{V}^{2})(t,x).
\end{eqnarray*}%
Summarizing, we have for all $(s,y),(t,x)\in Q_{0},$%
\begin{equation}
|\nabla G(s,y)|^{2}+[\partial _{s}G(s,y)]^{2}\leq C\mathcal{M}_{t}\mathcal{M}%
_{x}(|g|_{V}^{2})(t,x)|_{V}^{2}.  \notag
\end{equation}%
Therefore (\ref{2.10}) follows and we showed that (\ref{2.7}) holds in the
first and second case.

6$^{0}$. Now we show that (\ref{2.7}) in the case (2)-(1)  and (\ref{2.10})
in the case (3) imply (\ref{2.6}). Let $\varphi $ be a continuos function on 
$\mathbf{R}$ with all bounded derivatives such that $0\leq \varphi \leq
1,\varphi (s)=0$ if $-8\leq s,\varphi (s)=1$ if $s\leq -9$. Let 
\begin{eqnarray*}
g_{1}(s,y) &=&g(s,y)\varphi (s),(s,y)\in \mathbf{R}^{d+1}, \\
g_{2} &=&g-g_{1}\text{.}
\end{eqnarray*}%
Then by (\ref{maf2}),%
\begin{eqnarray*}
|Gg-(Gg)_{Q_{0}}| &\leq &|G(g_{1}+g_{2})-Gg_{1}|+|Gg_{1}-\left(
Gg_{1}\right) _{Q_{0}}| \\
+|\left( Gg_{1}\right) _{Q_{0}}-(Gg)_{Q_{0}}| &\leq &Gg_{2}+\left(
Gg_{2}\right) _{Q_{0}}+|Gg_{1}-\left( Gg_{1}\right) _{Q_{0}}|
\end{eqnarray*}%
and%
\begin{equation*}
\left( Gg\right) _{Q_{0}}^{\#}\leq (Gg_{1})_{Q_{0}}^{\#}+2(Gg_{2})_{Q_{0}}
\end{equation*}

Now, by (\ref{2.2}) and (\ref{2.3}), the required inequality (\ref{2.6})
follows from (\ref{2.7}), (\ref{eq25}) and (\ref{2.10}). The first assertion
of the proposition is proved.

7$^{0}.$ The estimates in Besov spaces follow immediately because%
\begin{equation*}
\left( \partial ^{{\scriptsize \alpha /2}}\mathcal{I}g\right) _{j}=\partial
^{{\scriptsize \alpha /2}}\mathcal{I}g_{j}
\end{equation*}%
and we have shown that%
\begin{eqnarray*}
&&\int_{a}^{b}\int_{\mathbf{R}^{d}}\bigg(\int_{a}^{t}\big\vert\varphi
_{j}\ast f(s,t,x)\big\vert_{V}^{2}ds\bigg)^{p/2}dxdt \\
&&\qquad =\int_{a}^{b}\int_{\mathbf{R}^{d}}\bigg(\int_{a}^{t}\big\vert\big(%
\partial ^{{\scriptsize \alpha /2}}\mathcal{I}g(s,t,x)\big)_{j}\big\vert%
_{V}^{2}ds\bigg)^{p/2}dxdt \\
&&\qquad =\int_{a}^{b}\int_{\mathbf{R}^{d}}\bigg(\int_{a}^{t}\big\vert%
\partial ^{{\scriptsize \alpha /2}}\mathcal{I}g_{j}(s,t,x)\big\vert_{V}^{2}ds%
\bigg)^{p/2}dxdt \\
&&\qquad \leq C\int_{a}^{b}|g_{j}(s,\cdot )|_{V,p}^{p}ds.
\end{eqnarray*}

The proposition is proved.

\end{document}